\documentclass{amsart}
\usepackage{graphicx, color}
\usepackage{amscd}
\usepackage{amsmath,empheq}
\usepackage{amsfonts}
\usepackage{amssymb}
\usepackage{mathrsfs}
\usepackage[all]{xy}
\newtheorem{theorem}{Theorem}
\newtheorem{ex}{Example}
\newtheorem{lemma}[theorem]{Lemma}
\newtheorem{prop}[theorem]{Proposition}
\newtheorem{remark}{Remark}
\newtheorem{corollary}[theorem]{Corollary}
\newtheorem{claim}{Claim}

\newenvironment{proof-sketch}{\noindent{\bf Sketch of Proof}\hspace*{1em}}{\qed\bigskip}

\everymath{\displaystyle}
\newcommand{\RR}{\mathbb R}
\newcommand{\NN}{\mathbb N}

\newcommand{\ZZ}{\mathbb Z}

\renewcommand{\leq}{\leqslant}

\renewcommand{\geq}{\geqslant}

\begin{document}
\title[Nonlinear nonhomogeneous boundary value problems]{Nonlinear nonhomogeneous boundary value problems with competition phenomena}
\author[N.S. Papageorgiou]{Nikolaos S. Papageorgiou}
\address[Nikolaos S. Papageorgiou]{National Technical University, Department of Mathematics,
				Zografou Campus, Athens 15780, Greece \& Institute of Mathematics, Physics and Mechanics, Jadranska 19, 1000 Ljubljana, Slovenia}
\email{\tt npapg@math.ntua.gr}
\author[V.D. R\u{a}dulescu]{Vicen\c{t}iu D. R\u{a}dulescu}
\address[Vicen\c{t}iu D. R\u{a}dulescu]{Institute of Mathematics, Physics and Mechanics, Jadranska 19, 1000 Ljubljana, Slovenia \& Faculty of Applied Mathematics, AGH University of Science and Technology, al. Mickiewicza 30, 30-059 Krak\'ow, Poland}
\email{\tt vicentiu.radulescu@imar.ro}
\author[D.D. Repov\v{s}]{Du\v{s}an D. Repov\v{s}}
\address[Du\v{s}an D. Repov\v{s}]{Faculty of Education and Faculty of Mathematics and Physics, University of Ljubljana,  \& Institute of Mathematics, Physics and Mechanics, Jadranska 19, 1000 Ljubljana, SloveniaSI-1000 Ljubljana, Slovenia}
\email{\tt dusan.repovs@guest.arnes.si}
\begin{abstract}
We consider a nonlinear boundary value problem driven by a nonhomogeneous differential operator. The problem exhibits competing nonlinearities with a superlinear (convex) contribution coming from the reaction term and a sublinear (concave) contribution coming from the parametric boundary (source) term. We show that for all small parameter values $\lambda>0$, the problem has at least five nontrivial smooth solutions, four of constant sign and one nodal. We also produce extremal constant sign solutions and determine their monotonicity and continuity properties as the parameter $\lambda>0$ varies. In the semilinear case we produce a sixth nontrivial solution but without any sign information. Our approach uses variational methods together with truncation and perturbation techniques, and Morse theory.
\end{abstract}
\maketitle

\section{Introduction}

Let $\Omega\subseteq\RR^N$ be a bounded domain with a $C^2$-boundary $\partial\Omega$. In this paper we study the following nonlinear, nonhomogeneous elliptic problem
\begin{equation}
	\left\{\begin{array}{ll}
		-{\rm div}\, a(Du(z))=f(z,u(z))&\mbox{in}\ \Omega,\\
		\frac{\partial u}{\partial n_a}=\lambda\beta(z,u)&\mbox{on}\ \partial\Omega.
	\end{array}\right\}\tag{$P_{\lambda}$} \label{eqP}
\end{equation}

In this problem $a:\RR^N\rightarrow\RR^N$ is a strictly monotone, continuous map which satisfies certain other regularity and growth conditions, listed in hypotheses $H(a)$ in Section 2. These hypotheses are general enough to incorporate in our framework several differential operators of interest, such as, e.g., the $p$-Laplacian. The reaction term $f(z,x)$ is a Carath\'eodory function (that is, for all $x\in\RR$, $z\mapsto f(z,x)$ is measurable and for almost all $z\in\Omega$, $x\mapsto f(z,x)$ is continuous) which satisfies the well-known Ambrosetti-Rabinowitz condition (AR-condition for short) in the $x$-variable, hence exhibiting $(p-1)$-superlinear growth near $\pm\infty$. In the boundary condition, $\frac{\partial u}{\partial n_a}$ denotes the generalized normal derivative corresponding to the differential operator $u\mapsto {\rm div}\, a(Du)$ and is defined by
$$\frac{\partial u}{\partial n_a}=(a(Du),n)_{\RR^N}\ \mbox{for all}\ u\in W^{1,p}(\Omega),$$
with $n(\cdot)$ being the outward unit normal on $\partial\Omega$. This kind of generalized normal derivative is dictated by the nonlinear Green's identity (see Gasinski and Papageorgiou \cite[p. 210]{13}  and it was also used by Lieberman \cite{21}). The boundary function $\beta(z,x)$ is continuous on $\partial\Omega\times\RR$ and it satisfies certain other regularity and growth conditions listed in hypotheses $H(\beta)$ in Section 3. In fact, $\beta(z,\cdot)$ exhibits strict $(p-1)$-sublinear growth near $\pm\infty$. So, we see that problem \eqref{eqP} has competing nonlinearities. We refer to a convex (superlinear) input coming from the reaction term $f(z,x)$ and a concave (sublinear) input resulting from the source (boundary) term.

The study of problems with competition phenomena was initiated with the seminal paper of Ambrosetti, Brezis and Cerami \cite{2} for semilinear Dirichlet equations. In their work both competing nonlinearities appear in the reaction term $f(z,x)$, which has the form
$$f(z,x)=f(x)=\lambda|x|^{q-2}x+|x|^{r-2}x\ \mbox{for all}\ x\in\RR,$$
with $\lambda>0,\ 1<q<2<r\leq 2^*=\left\{\begin{array}{ll}
	\frac{2N}{N-2}&\mbox{if}\ N\geq 3\\
	+\infty&\mbox{if}\ N=1,2
\end{array}\right.$. Since then there has been a lot of work in this direction, extending the results of \cite{2} to nonlinear equations. In contrast, in the present paper the concave term comes from the boundary condition. The study of such problems is still lagging behind. In this direction, there are only the semilinear works of Furtado and Ruviaro \cite{11}, Garcia Azorero, Peral Alonso and Rossi \cite{12} and Hu and Papageorgiou \cite{20}.

In this paper, we prove a multiplicity theorem which says that for small values of the parameter $\lambda>0$, the problem has at least five nontrivial smooth solutions, four of constant sign and one nodal. We also show the existence of extremal constant sign solutions, that is, a smallest positive solution $u^*_{\lambda}$ and a biggest negative solution $v^*_{\lambda}$, and we investigate the monotonicity and continuity properties of the maps $\lambda\mapsto u^*_{\lambda}$ and $\lambda\mapsto v^*_{\lambda}$. Finally, in the semilinear case, we generate a sixth nontrivial smooth solution (without being able to provide any sign information).

Our approach uses variational methods based on the critical point theory, combined with suitable truncation-perturbation and comparison techniques, and Morse theory.

\section{Preliminaries -- Hypotheses}

In this section we present the main mathematical tools which we will use in the sequel and we prove some auxiliary results which will be needed later.
In this section we also fix our notation and we have gathered all the hypotheses on the data of problem \eqref{eqP} which will be used to prove our results. We also state the main results of this work, in order for the reader to have a feeling of what is achieved in this paper.

Let $X$ be a Banach space and let $X^*$ be its topological dual. By $\left\langle \cdot,\cdot\right\rangle$ we denote the duality brackets for the pair $(X^*,X)$. Let $\varphi\in C^1(X,\RR)$. We say that $\varphi$ satisfies the ``Cerami condition'' (the ``C-condition'' for short), if the following property holds:
\begin{center}
``Every sequence $\{u_n\}_{n\geq 1}\subseteq X$ such that $\{\varphi(u_n)\}_{n\geq 1}\subseteq\RR$ is bounded and
$$(1+||u_n||)\varphi'(u_n)\rightarrow 0\ \mbox{in}\ X^*\ \mbox{as}\ n\rightarrow\infty,$$
admits a strongly convergent subsequence''.
\end{center}

This compactness-type condition on the functional $\varphi$ is needed in the critical point theory because the ambient space need not be locally compact (being in general infinite dimensional). Using this compactness-type condition, one can prove a deformation theorem describing the change of the topological structure of the sublevel sets of $\varphi$ along the negative gradient or pseudogradient flow. The deformation theorem leads to the minimax theory of the critical values of $\varphi$. Prominent in that theory is the so-called ``mountain pass theorem'' due to Ambrosetti and Rabinowitz \cite{3}. Here we state it in a slightly more general form (see, for example, Gasinski and Papageorgiou \cite[p. 648]{13}).

\begin{theorem}\label{th1}
	Let $X$ be a Banach space. Suppose that $\varphi\in C^1(X,\RR)$  satisfies the C-condition, and $u_0,u_1\in X$ satisfy $||u_1-u_0||>\rho>0$
	$$\max\{\varphi(u_0),\varphi(u_1)\}<\inf[\varphi(u):||u-u_0||=\rho]=m_{\rho}.$$
	Let $c=\inf\limits_{\gamma\in\Gamma}\max\limits_{0\leq t\leq 1}\varphi(\gamma(t))$ with $\Gamma=\{\gamma\in C([0,1],X):\gamma(0)=u_0,\gamma(1)=u_1\}$. Then $c\geq m_{\rho}$ and $c$ is a critical value of $\varphi$ (that is, there exists $\hat{u}\in X$ such that $\varphi'(\hat{u})=0$ and $\varphi(\hat{u})=c$).
\end{theorem}

Let $\vartheta\in C^1(0,\infty)$ with $\vartheta(t)>0$ for all $t>0$ and assume that
\begin{equation}\label{eq1}
	0<\hat{c}\leq\frac{\vartheta'(t)t}{\vartheta(t)}\leq c_0\ \mbox{and}\ c_1t^{p-1}\leq\vartheta(t)\leq c_2(1+t^{p-1})\ \mbox{for all}\ t>0,\ \mbox{with}\ c_1,c_2>0.
\end{equation}

The hypotheses on the map $y\mapsto a(y)$ involved in the definition of the differential operator in problem \eqref{eqP}, are the following:

\smallskip
$H(a):$ $a(y)=a_0(|y|)y$ for all $y\in\RR^N$, with $a_0(t)>0$ for all $t>0$ and
\begin{itemize}
	\item[(i)] $a_0\in C^1(0,+\infty),\ t\mapsto a_0(t)t$ is strictly increasing on $(0,+\infty)$, $a_0(t)t\rightarrow 0^+$ as $t\rightarrow 0^+$ and
	$$\lim\limits_{t\rightarrow 0^+}\frac{a'_0(t)t}{a_0(t)}>-1;$$
	\item[(ii)] there exists $c_3>0$ such that
	$$|\nabla a(y)|\leq c_3\frac{\vartheta(|y|)}{|y|}\ \mbox{for all}\ y\in \RR^N\backslash\{0\};$$
	\item[(iii)] $(\nabla a(y)\xi,\xi)_{\RR^N}\geq\frac{\vartheta(|y|)}{|y|}|\xi|^2$ for all $y\in\RR^N\backslash\{0\}$ and all $\xi\in\RR^N$;
	\item[(iv)] if $G_0(t)=\int^t_0a_0(s)sds$ for $t>0$, then there exists $\tau\in(1,p)$ such that
	\begin{eqnarray*}
		&&t\mapsto G_0(t^{1/\tau})\ \mbox{is convex},\ 0\leq\liminf\limits_{t\rightarrow 0^+}\frac{G_0(t)}{t^{\tau}}\leq\limsup\limits_{t\rightarrow 0^+}\frac{G_0(t)}{t^{\tau}}\leq\tilde{c},\\
		&&c_4t^p\leq a_0(t)t-\tau G_0(t)\ \mbox{for all}\ t>0\ \mbox{and some}\ c_4>0,\\
		&&-\bar{c}\leq pG_0(t)-a_0(t)t^2\ \mbox{for all}\ t>0\ \mbox{and some}\ \bar{c}>0.
	\end{eqnarray*}
\end{itemize}
\begin{remark}
	Hypotheses $H(a)(i),(ii),(iii)$ were motivated by the nonlinear regularity theory of Lieberman \cite{21} and the nonlinear maximum principle of Pucci and Serrin \cite[pp. 111, 120]{33}. Hypothesis $H(a)(iv)$ serves the particular needs of our problem. However, this is a mild restriction and it is satisfied in all cases of interest, as the examples below show.
\end{remark}

	Clearly hypotheses $H(a)$ imply that $G_0(\cdot)$ is strictly convex and strictly increasing. We set
	$$G(y)=G_0(|y|)\ \mbox{for all}\ y\in\RR^N.$$

Then $G(\cdot)$ is convex, $G(0)=0$, and
$$\nabla G(0)=0,\nabla G(y)=G'_0(|y|)\frac{y}{|y|}=a_0(|y|)y=a(y)\ \mbox{for all}\ y\in\RR^N\backslash\{0\}.$$

So, $G(\cdot)$ is the primitive of the map $a(\cdot)$. Using the convexity of $G(\cdot)$ and  $G(0)=0$, we have
\begin{equation}\label{eq2}
	G(y)\leq(a(y),y)_{\RR^N}\ \mbox{for all}\ y\in \RR^N.
\end{equation}

Hypotheses $H(a)(i),(ii),(iii)$ and (\ref{eq1}), lead to the following lemma which summarizes the main properties of the map $y\mapsto a(y)$.
\begin{lemma}\label{lem2}
	If hypotheses $H(a)(i),(ii),(iii)$ hold, then
	\begin{itemize}
		\item[(a)] $y\mapsto a(y)$ is strictly monotone, continuous, hence also maximal monotone;
		\item[(b)] $|a(y)|\leq c_4(1+|y|^{p-1})$ for all $y\in\RR^N$ and some $c_4>0$;
		\item[(c)] $(a(y),y)_{\RR^N}\geq\frac{c_1}{p-1}|y|^p$ for all $y\in\RR^N.$
	\end{itemize}
\end{lemma}

This lemma and (\ref{eq2}) lead to the following growth estimates for the primitive $G(\cdot)$.
\begin{corollary}\label{cor3}
	If hypotheses $H(a)(i),(ii),(iii)$ hold, then $\frac{c_1}{p(p-1)}|y|^p\leq G(y)\leq c_5(1+|y|^p)$ for all $y\in\RR^N$ and some $c_5>0$.
\end{corollary}

The examples which follow illustrate that hypotheses $H(a)$ cover many interesting cases.
\begin{ex}
The following maps satisfy hypotheses $H(a)$:
\begin{itemize}
	\item[(a)] $a(y)=|y|^{p-2}y$ with $1<p<\infty$.
	
	This map corresponds to the $p$-Laplacian differential operator defined by
	$$\Delta_pu={\rm div}\,(|Du|^{p-2}Du)\ \mbox{for all}\ u\in W^{1,p}(\Omega).$$
	\item[(b)] $a(y)=|y|^{p-2}y+|y|^{\tau-2}y$ with $1<\tau<p$
	
	This map corresponds to the $(p,\tau)$-differential operator defined by
	$$\Delta_pu+\Delta_{\tau}u\ \mbox{for all}\ u\in W^{1,p}(\Omega).$$

	Such differential operators arise in problems of mathematical physics. We mention the works of Benci, D'Avenia, Fortunato and Pisani \cite{4} (in quantum physics) and Cherfils and Ilyasov \cite{6} (in plasma physics). Recently, there have been some existence and multiplicity results for such equations. We mention the works of Cingolani and Degiovanni \cite{7}, Gasinski and Papageorgiou \cite{15}, Marano, Mosconi and Papageorgiou \cite{22}, Mugnai and Papageorgiou \cite{24}, Papageorgiou and R\u adulescu \cite{26, 27}, Papageorgiou, R\u adulescu and Repov\v{s} \cite{31}, Sun \cite{34}, and Sun, Zhang and Su \cite{35}.
	\item[(c)] $a(y)=(1+|y|^2)^{\frac{p-2}{2}y}$ with $1<p<\infty$.
	
	This map corresponds to the generalized $p$-mean curvature differential operator defined by
	$${\rm div}\, ((1+|Du|^2)^{\frac{p-2}{2}}Du)\ \mbox{for all}\ u\in W^{1,p}(\Omega).$$
	\item[(d)] $a(y)=|y|^{p-2}y\left[1+\frac{1}{1+|y|^p}\right]$ for all $u\in W^{1,p}(\Omega)$.
	
	This map corresponds to the differential operator
	$$\Delta_pu+{\rm div}\,\left(\frac{|Du|^{p-2}Du}{1+|Du|^p}\right)\ \mbox{for all}\ u\in W^{1,p}(\Omega).$$
	\item[(e)] $a(y)=|y|^{p-2}y+\ln(1+|y|^2)y$ with $1<p<\infty$.
	
	This map corresponds to the differential operator
	$$\Delta_pu+{\rm div}\,(\ln(1+|Du|^2)Du)\ \mbox{for all}\ u\in W^{1,p}(\Omega).$$
\end{itemize}
\end{ex}

We will use the Sobolev space $W^{1,p}(\Omega)$, the Banach space $C^1(\overline{\Omega})$ and the boundary Lebesgue spaces $L^q(\partial\Omega);1\leq q\leq\infty$. The Sobolev space $W^{1,p}(\Omega)$ is a Banach space for the norm
$$||u||=[\,||u||^p_p+||Du||^p_p\,]^{1/p}\ \mbox{for all}\ u\in W^{1,p}(\Omega).$$

The Banach space $C^1(\overline{\Omega})$ is an ordered Banach space with positive (order) cone
$$C_+=\{u\in C^1(\overline{\Omega}):u(z)\geq 0\ \mbox{for all}\ z\in\overline{\Omega})\}.$$

This cone has a nonempty interior given by
$${\rm int}\, C_+=\{u\in C_+:u(z)>0\ \mbox{for all}\ z\in\Omega,\ \frac{\partial u}{\partial n}\big|_{\partial\Omega\cap u^{-1}(0)}<0\ \mbox{if}\  \partial\Omega\cap u^{-1}(0)\not=\emptyset\}.$$
This cone contains the open set
$$D_+=\{u\in C_+:u(z)>0\ \mbox{for all}\ z\in\overline{\Omega}\}.$$

In fact, note that $D_+$ is the interior of $C_+$ when $C^1(\overline\Omega)$ is furnished with the relative $C(\overline\Omega)$-topology.

On $C^1(\overline\Omega)$ the $C^1(\overline\Omega)$-norm topology is stronger than the $C(\overline\Omega)$-norm topology. Therefore we have
$$D_+\subseteq{\rm int}\, C_+. $$

On $\partial\Omega$ we consider the $(N-1)$-dimensional Hausdorff (surface) measure $\sigma(\cdot)$. Using this measure, we can define in the usual way the Lebesgue spaces $L^q(\partial\Omega)$, $1\leq q\leq\infty$. From the theory of Sobolev spaces we know that there exists a unique continuous linear map $\gamma_0:W^{1,p}(\Omega)\rightarrow L^{\tau}(\partial\Omega)$, $\tau=\frac{Np-p}{N-p}$ if $p<N$, and $\tau\geq 1$ if $N\leq p$, known as the ``trace map'', such that
$$\gamma_0(u)=u|_{\partial\Omega}\ \mbox{for all}\ u\in W^{1,p}(\Omega)\cap C(\overline{\Omega}).$$

So, we can understand the trace map as an expression of the ``boundary values'' of a Sobolev function. We know that
$${\rm im}\,\gamma_0=W^{\frac{1}{p'},p}(\partial\Omega),\quad\mbox{where}\ \frac{1}{p}+\frac{1}{p'}=1\ \mbox{and}\ {\rm ker}\,\gamma_0=W^{1,p}_0(\Omega).$$

The trace map $\gamma_0$ is compact into $L^q(\partial\Omega)$ for all $q\in\left[1,\frac{Np-p}{N-p}\right)$ when $1<p<N$ and for all $q\geq 1$, when $p\geq N$. In the sequel, for the sake of notational simplicity we drop the use of the map $\gamma_0$. All restrictions of Sobolev functions on $\partial\Omega$ are understood in the sense of traces.

Introducing some more notation, for every $x\in\RR$, we set $x^{\pm}=\max\{\pm x,0\}$. Then for $u\in W^{1,p}(\Omega)$ we define $u^{\pm}(\cdot)=u(\cdot)^{\pm}$ and have
$$u=u^+-u^-,\ |u|=u^++u^-,\ u^+,\ u^-\in W^{1,p}(\Omega).$$

By $|\cdot|_N$ we denote the Lebesgue measure on $\RR^N$ and if $g:\Omega\times\RR\rightarrow\RR$ is a measurable function (for example, a Carath\'eodory function), then we define the Nemytskii map corresponding to $g$
$$N_g(u)(\cdot)=g(\cdot,u(\cdot))\ \mbox{for all}\ u\in W^{1,p}(\Omega).$$

Let $A:W^{1,p}(\Omega)\rightarrow W^{1,p}(\Omega)^*$ be the nonlinear map defined by
\begin{equation}\label{eq3}
	\left\langle A(u),h\right\rangle=\int_{\Omega}(a(Du),Dh)_{\RR^N}dz\ \mbox{for all}\ u,h\in W^{1,p}(\Omega).
\end{equation}

The next proposition establishes the main properties of this map. It is a special case of Proposition 3.5 in Gasinski and Papageorgiou \cite{14}.
\begin{prop}\label{prop4}
	Assume that hypotheses $H(a)(i),(ii),(iii)$ hold and that $A:W^{1,p}(\Omega)\rightarrow W^{1,p}(\Omega)^*$ is the nonlinear map defined by (\ref{eq3}). Then $A$ is bounded (that is,  maps bounded sets to bounded sets), continuous, monotone (hence also maximal monotone) and of type $(S)_+$ (that is, if $u_n\stackrel{w}{\rightarrow}u$ in $W^{1,p}(\Omega)$ and $\limsup\limits_{n\rightarrow\infty}\left\langle A(u_n),u_n-u\right\rangle\leq 0$, then $u_n\rightarrow u$ in $W^{1,p}(\Omega)$).
\end{prop}

Next, consider a Carath\'eodory function $f_0:\Omega\times\RR\rightarrow\RR$ and a function $\beta_0\in C(\partial\Omega\times\RR)\cap C^{0,\alpha}_{loc}(\partial\Omega\times\RR)$ with $\alpha\in\left(0,1\right]$ such that
$$|f_0(z,x)|\leq a_0(z)(1+|x|^{r-1})\ \mbox{for almost all}\ z\in\Omega\ \mbox{and all}\ x\in\RR$$
with $a_0\in L^{\infty}(\Omega)_+,p\leq r<p^*=\left\{\begin{array}{ll}
	\frac{Np}{N-p}&\mbox{if}\ p<N\\
	+\infty&\mbox{if}\ p\geq N
\end{array}\right.$ and
$$|\beta_0(z,x)|\leq c_6(1+|x|^{q-1})\ \mbox{for all}\ (z,x)\in\partial\Omega\times\RR,$$
with $c_6>0,1<q<p$. We set
$$F_0(z,x)=\int^x_0f_0(z,s)ds,B_0(z,x)=\int^x_0\beta_0(z,s)ds\ \mbox{for all}\ (z,x)\in\partial\Omega\times\RR,$$
and consider the $C^1$-functional $\varphi_0:W^{1,p}(\Omega)\rightarrow\RR$ defined by
$$\varphi_0(u)=\int_{\Omega}G(Du)dz-\int_{\Omega}F_0(z,u)dz-\int_{\partial\Omega}B_0(z,u)d\sigma\ \mbox{for all}\ u\in W^{1,p}(\Omega).$$

From Papageorgiou and R\u adulescu \cite{30} and \cite{28} (the case of the $p$-Laplacian) we obtain the following property.
\begin{prop}\label{prop5}
	Assume that $u_0\in W^{1,p}(\Omega)$ is a local $C^1(\overline{\Omega})$-minimizer of the functional $\varphi_0$, that is, there exists $\rho_0>0$ such that
	$$\varphi_0(u_0)\leq\varphi_0(u_0+h)\ \mbox{for all}\ h\in C^1(\overline{\Omega}),\ ||h||_{C^1(\overline{\Omega})}\leq\rho_0.$$
	Then $u_0\in C^{1,\mu}(\overline{\Omega})$ with $\mu\in(0,1)$ and $u_0$ is also a local $W^{1,p}(\Omega)$-minimizer of $\varphi_0$, that is, there exists $\rho_1>0$ such that
	$$\varphi_0(u_0)\leq\varphi_0(u_0+h)\ \mbox{for all}\ h\in W^{1,p}(\Omega),\ ||h||\leq\rho_1.$$
\end{prop}

Next, let us recall some basic definitions and facts from Morse theory (critical groups) which we will need later.

Given a Banach space $X$, a function $\varphi\in C^1(X,\RR)$ and $c\in\RR$, we introduce the following sets:
\begin{eqnarray*}
	&&\varphi^c=\{u\in X:\varphi(u)\leq c\}\ (\mbox{the sublevel set of}\ \varphi\ \mbox{at the level}\ c),\\
	&&K_{\varphi}=\{u\in X:\varphi'(u)= c\}\ (\mbox{the critical set of}\ \varphi),\\
	&&K_{\varphi}^c=\{u\in K_{\varphi}:\varphi(u)=c\}\ (\mbox{the critical set of}\ \varphi\ \mbox{at the level}\ c).
\end{eqnarray*}

Let $(Y_1,Y_2)$ be a topological pair such that $Y_2\subseteq Y_1\subseteq X$. By $H_k(Y_1,Y_2),\ k\in\NN_0$, we denote the $k$th relative singular homology group for the topological pair $(Y_1,Y_2)$ with integer coefficients. The critical groups of $\varphi$ at an isolated point $u\in K^c_{\varphi}$ are defined by
$$C_k(\varphi,u)=H_k(\varphi^c\cap U,\varphi^c\cap U\backslash\{0\})\ \mbox{for all}\ k\in\NN_0.$$
Here, $U$ is a neighborhood of $u$ such that $K_{\varphi}\cap\varphi^c\cap U=\{u\}$. The excision property of singular homology implies that the above definition of critical groups is independent of the choice of the neighborhood $U$ of $u$.

Suppose that $\varphi$ satisfies the $C$-condition and that $\inf\varphi(K_{\varphi})>-\infty$. Let $c<\inf \varphi(K_{\varphi})$. The critical groups of $\varphi$ at infinity are defined by
$$C_k(\varphi,\infty)=H_k(X,\varphi^c)\ \mbox{for all}\ k\in\NN_0.$$

The second deformation theorem (see, for example, Gasinski and Papageorgiou \cite[p. 628]{13}) implies that this definition is independent of the level $c<\inf\varphi(K_{\varphi})$.

Suppose that $\varphi\in C^1(X,\RR)$ satisfies the $C$-condition and that $K_{\varphi}$ is finite. We define
\begin{eqnarray*}
	&&M(t,u)=\sum\limits_{k\in\NN_0}{\rm rank}\, C_k(\varphi,u)t^k\ \mbox{for all}\ t\in\RR\ \mbox{and all}\ u\in K_{\varphi},\\
	&&P(t,\infty)=\sum\limits_{k\in\NN_0}{\rm rank}\, C_k(\varphi,\infty)t^{k}\ \mbox{for all}\ t\in\RR.
\end{eqnarray*}

Then the Morse relation says that
\begin{equation}\label{eq4}
	\sum\limits_{u\in K_{\varphi}}M(t,u)=P(t,\infty)+(1+t)Q(t)\ \mbox{for all}\ t\in\RR,
\end{equation}
with $Q(t)=\sum\limits_{k\in\NN_0}\beta_kt^k$ being  a formal series in $t\in\RR$ with nonnegative integer coefficients $\beta_k$.

Next, we state a strong comparison principle. Our proof uses ideas from Guedda and V\'eron \cite{17}, who were the first to prove a strong comparison principle for the Dirichlet $p$-Laplacian. Recall that $n(\cdot)$ denotes the outward unit normal on $\partial\Omega$.
\begin{prop}\label{prop6}
	Assume that hypotheses $H(a)(i),(ii),(iii)$ hold, $u_1,u_2\in C^1(\overline{\Omega}),$ $g_1,g_2\in L^{\infty}(\Omega)$, $u_1(z)\leq u_2(z)$ for all $z\in\overline{\Omega}$, and
	\begin{eqnarray*}
		&&g_1(z)\leq g_2(z)\ \mbox{for almost all}\ z\in\Omega,\ g_1\not\equiv g_2,\\
		&&-{\rm div}\, a(Du_1(z))=g_1(z)\ \mbox{for almost all}\ z\in\Omega,\left.\frac{\partial u_1}{\partial n}\right|_{\partial\Omega}>0\ \mbox{or}\ \left.\frac{\partial u_1}{\partial n}\right|_{\partial\Omega}<0,\\
		&&-{\rm div}\, a(Du_2(z))=g_2(z)\ \mbox{for almost all}\ z\in\Omega,\left.\frac{\partial u_2}{\partial n}\right|_{\partial\Omega}>0\ \mbox{or}\ \left.\frac{\partial u_2}{\partial n}\right|_{\partial\Omega}<0.
	\end{eqnarray*}
	Then $(u_2-u_1)(z)>0$ for all $z\in\Omega$ and $\frac{\partial(u_2-u_1)}{\partial n}(z_0)<0$ for all $z_0\in\Sigma_0=\{z\in\partial\Omega:u_2(z)=u_1(z)\}$.
\end{prop}
\begin{proof}
	By hypothesis we have
	\begin{equation}\label{eq5}
		-{\rm div}\, (a(Du_2(z))-a(Du_1(z)))=g_2(z)-g_1(z)\geq 0\ \mbox{for almost all}\ z\in\Omega\,.
	\end{equation}
	
	Let $a=(a_k)^N_{k=1}$ with $a_k:\RR^N\rightarrow\RR$ for every $k\in\{1,\ldots,N\}$. Using the mean value theorem, we have
	\begin{equation}\label{eq6}
		a_k(y)-a_k(y')=\overset{N}{\underset{\mathrm{i=1}}\sum}\int^1_0\frac{\partial a_k}{\partial y_i}(y'+t(y-y'))(y_i-y'_i)dt
	\end{equation}
	for all $y=(y_i)^N_{i=1},y'=(y'_i)^N_{i=1}\in\RR^N$ and all $k\in\{1,\ldots,N\}$.
	
	We introduce the following coefficient functions
	\begin{equation}\label{eq7}
		c_{k,i}(z)=\int^1_0\frac{\partial a_k}{\partial y_i}(Du_1(z)+t(Du_2(z)-Du_1(z))).
	\end{equation}
	
	Using these coefficients, we introduce the following linear differential operator
	\begin{equation}\label{eq8}
		L(v)=-{\rm div}\, \left(\overset{N}{\underset{\mathrm{i=1}}\sum}c_{k,i}(z)\frac{\partial v}{\partial z_i}\right)=-\overset{N}{\underset{\mathrm{k,i=1}}\sum}\frac{\partial}{\partial z_k}\left(c_{k,i}(z)\frac{\partial v}{\partial z_i}\right).
	\end{equation}
	
	Let $v=u_2-u_1$. Then $v\neq 0$ (recall that $g_1\not\equiv g_2$) and from (\ref{eq5})--(\ref{eq8}) we have
	\begin{equation}\label{eq9}
		L(v)(z)=g_2(z)-g_1(z)\geq 0\ \mbox{for almost all}\ z\in\Omega.
	\end{equation}
	
	By hypothesis, we have $\left.\frac{\partial u_1}{\partial n}\right|_{\partial\Omega},\left.\frac{\partial u_2}{\partial n}\right|_{\partial\Omega}>0$ or $<0$. So, for small $\delta>0$ we have
	\begin{equation}\label{eq10}
		|D((1-t)u_1(z)+tu_2(z))|\geq\eta>0\ \mbox{for all}\ z\in\overline{\Omega}_{\delta},
	\end{equation}
	with $\Omega_{\delta}=\{z\in\Omega:d(z,\partial\Omega)<\delta\}$. It follows from (\ref{eq8}) and (\ref{eq10})  that the operator $L$ is strictly elliptic on $\Omega_{\delta}$.
	
	Suppose that $u_1|_{\Omega_{\delta}}=u_2|_{\Omega_{\delta}}$. Then $g_1(z)=g_2(z)$ for almost all $z\in\Omega_{\delta}$. We consider a function $\vartheta\in C^1(\overline{\Omega})$ such that
	\begin{equation}\label{eq11}
		\vartheta(z)>0\ \mbox{for all}\ z\in\Omega,\ \vartheta|_{\partial\Omega}=0,\ \vartheta|_{\Omega\backslash\Omega_{\delta}}\equiv 1.
	\end{equation}
	
	We have
	\begin{eqnarray*}
		\int_{\Omega}g_1\vartheta dz=\left\langle A(u_1),\vartheta\right\rangle&=&\int_{\Omega_{\delta}}(a(Du_1),D\vartheta)_{\RR^N}dz\ (\mbox{see (\ref{eq11})})\\
		&=&\int_{\Omega_{\delta}}(a(Du_2),D\vartheta)_{\RR^N}dz\ (\mbox{recall that}\ u_1|_{\Omega_{\delta}}=u_2|_{\Omega_{\delta}})\\
		&=&\int_{\Omega}g_2\vartheta dz\ (\mbox{see (\ref{eq11})}),
	\end{eqnarray*}
	which is in contradiction with the hypothesis that $g_1\not\equiv g_2$ (recall $\vartheta>0$, see (\ref{eq11})). So, we have $u_2-u_1\in C_+\backslash\{0\}$.
	
	Then from (\ref{eq9}) and the strong maximum principle (see, for example, Gasinski and Papageorgiou \cite[p. 738]{13}), we derive
	\begin{equation}\label{eq12}
		(u_2-u_1)(z)>0\ \mbox{for all}\ z\in\Omega_{\delta}\ \mbox{and}\ \left.\frac{\partial(u_2-u_1)}{\partial_n}\right|_{\Sigma_0}<0.
	\end{equation}
	
It follows	from (\ref{eq12}) that the set $S=\{z\in\Omega:u_1(z)=u_2(z)\}$ is compact. Hence Corollary 8.23, p. 215, of Motreanu, Motreanu and Papageorgiou \cite{23}, implies that
	$$(u_2-u_1)(z)>0\ \mbox{for all}\ z\in\Omega\ \mbox{and}\ \left.\frac{\partial(u_2-u_1)}{\partial n}\right|_{\Sigma_0}<0.$$
This completes the proof.
\end{proof}

\begin{remark}
	Consider the following order cone in $C^1(\overline{\Omega})$:
	$$\hat{C}_+=\{y\in C^1(\overline{\Omega}):y(z)\geq 0\ \mbox{for all}\ z\in\overline{\Omega},\left.\ \frac{\partial y}{\partial n}\right|_{\Sigma_0}\leq 0\}\,,$$
	where $\Sigma_0=\{z\in\partial\Omega:y(z)=0\}$. This cone has a nonempty interior given by
	$${\rm int}\,\hat{C}_+=\{y\in\hat{C}_+:y(z)>0\ \mbox{for all}\ z\in\Omega,\left.\ \frac{\partial y}{\partial n}\right|_{\Sigma_0}<0\}.$$
	Then Proposition \ref{prop6} says that $u_2-u_1\in int\hat{C}_+$.
\end{remark}

We will also use the next proposition, which essentially produces an equivalent norm for the Sobolev space $W^{1,p}(\Omega)$. The result is stated in a more general form than the one we will need, because we believe that in this form it is of independent interest and can be used in other circumstances.
\begin{prop}\label{prop7}
	Assume that $\beta\in L^{\infty}(\partial\Omega),\beta(z)\geq 0$ for almost all $z\in\partial\Omega,\beta\not\equiv 0$, $1\leq q\leq\frac{Np-p}{N-p}$ if $p<N$, and $1\leq p$ if $N\leq p$, and $|u|=||Du||_p+(\int_{\partial\Omega}\beta(z)|u|^qd\sigma)^{1/q}$ for all $u\in W^{1,p}(\Omega)$. Then we can find $0<c_7\leq c_8$ such that $c_7|u|\leq||u||\leq c_8|u|$ for all $u\in W^{1,p}(\Omega)$.
\end{prop}
\begin{proof}
	Note that
	\begin{eqnarray}\label{eq13}
		&|u|&\leq||Du||_p+||\beta||^{1/q}_{L^{\infty}(\partial\Omega)}||u||_{L^q(\partial\Omega)}\nonumber\\
		&&\leq||Du||_p+||\beta||_{L^{\infty}(\partial\Omega)}||\gamma_0||_{\mathcal{L}}||u||\nonumber\\
		&&\leq c_9||u||\ \mbox{for some}\ c_9>0.
	\end{eqnarray}
	
	Next we show that we find $c_{10}>0$ such that
	\begin{equation}\label{eq14}
		||u||_p\leq c_{10}|u|\ \mbox{for all}\ u\in W^{1,p}(\Omega).
	\end{equation}
	
	Suppose that (\ref{eq14}) is not true. Then we can find $\{u_n\}_{n\geq 1}\subseteq W^{1,p}(\Omega)$ such that
	$$||u_n||_p>n|u_n|\ \mbox{for all}\ n\in\NN.$$
	
	Normalizing in $L^p(\Omega)$ if necessary, we may assume that $||u_n||_p=1$ for all $n\in\NN$. Then
	\begin{eqnarray}\label{eq15}
		&&|u_n|<\frac{1}{n}\ \mbox{for all}\ n\in\NN,\nonumber\\
		&\Rightarrow&|u_n|\rightarrow 0\ \mbox{as}\ n\rightarrow\infty,\\
		&\Rightarrow&||Du_n||_p\rightarrow 0\ \mbox{as}\ n\rightarrow\infty,\nonumber\\
		&\Rightarrow&\{u_n\}_{n\geq 1}\subseteq W^{1,p}(\Omega)\ \mbox{is bounded}\ (\mbox{recall that}\ ||u_n||_p=1\ \mbox{for all}\ n\in\NN).\nonumber
	\end{eqnarray}
	
	Then by passing to a subsequence if necessary, we may assume that
	\begin{equation}\label{eq16}
		u_n\stackrel{w}{\rightarrow}u\ \mbox{in}\ W^{1,p}(\Omega),\ u_n\rightarrow u\ \mbox{in}\ L^p(\Omega)\ \mbox{and}\ u_n\stackrel{w}{\rightarrow}u\ \mbox{in}\ L^q(\partial\Omega)
	\end{equation}
	(here we use the continuity of the trace map). It follows from (\ref{eq15}), (\ref{eq16}) that
	\begin{eqnarray}\label{eq17}
		&&||Du||_p+\left(\int_{\partial\Omega}\beta(z)|u|^qd\sigma\right)^{1/q}\leq 0\ (\mbox{recall that}\ \beta\in L^{\infty}(\Omega))\\
		&\Rightarrow&u\equiv\xi\in\RR\,.\nonumber
	\end{eqnarray}
	
	If $\xi\neq 0$, then by virtue of (\ref{eq17}) we have
	$$0<|\xi|^q\int_{\partial\Omega}\beta(z)d\sigma\leq 0,$$
	 a contradiction. Hence $\xi=0$ and so from (\ref{eq16}) we have
	$$u_0\rightarrow 0\ \mbox{in}\ L^p(\Omega),$$
	which is a contradiction with the fact that $||u_n||_p=1$ for all $n\in\NN$. So, (\ref{eq14}) holds and this, combined with (\ref{eq13}), implies that the assertion of the proposition is true.
\end{proof}
\begin{remark}
	If $\beta\equiv 1$, then Proposition \ref{prop7} asserts that
	$$u\mapsto |u|=||Du||_p+||u||_{L^q(\partial\Omega)}$$
	with $q\in\left[1,\frac{Np-p}{N-p}\right]$ if $p<N$, and $q\geq 1$ if $N<p$, is an equivalent norm on the Sobolev space $W^{1,p}(\Omega)$ (see also Gasinski and Papageorgiou \cite{13}, Proposition 2.5.8, p. 218).
\end{remark}

Finally we present all the conditions on the other data of \eqref{eqP} (that is, for $f(z,x)$ and $\beta(z,x)$) which we will use to prove our results and then we have the statements of our main results.

We start with the following hypotheses on the reaction term $f(z,x)$.

\smallskip
$H(f):$ $f:\Omega\times\RR\rightarrow\RR$ is a Carath\'eordory function such that $f(z,0)=0$ for almost all $z\in\Omega$ and
\begin{itemize}
	\item[(i)] $|f(z,x)|\leq a(z)(1+|x|^{r-1})$ for almost all $z\in\Omega$ and all $x\in\RR$, with $a\in L^{\infty}(\Omega)_+$, $p<r<p^*$;
	\item[(ii)] if $F(z,x)=\int^x_0f(z,s)ds$, then there exist $\eta>p$ and $M>0$ such that
	\begin{center}
		$0<\eta F(z,x)\leq f(z,x)x$ for almost all $z\in\Omega$ and all $|x|\geq M$;\\
		$f(z,x)x\leq c^*_1|x|^r-c^*_2|x|^p$ for almost all $z\in\Omega$, all $|x|\geq M$, and some $c^*_1,c^*_2>0$;
	\end{center}
	\item[(iii)] $\lim\limits_{x\rightarrow 0}\frac{f(z,x)}{|x|^{p-2}x}=0$ uniformly for almost all $z\in\Omega$.
\end{itemize}
\begin{remark}
	Hypothesis $H(f)(ii)$ is the well-known Ambrosetti-Rabinowitz condition and it implies that
	\begin{eqnarray}\label{eq18}
		c_{11}|x|^{\eta}\leq F(z,x)\ \mbox{for almost all}\ z\in\Omega,\ \mbox{all}\ |x|\geq M\ \mbox{and some}\ c_{11}>0.
	\end{eqnarray}
From (\ref{eq18}) and hypothesis $H(f)(ii)$, we infer that for almost all $z\in\Omega$, $f(z,\cdot)$ is $(p-1)$-superlinear. It would be interesting to know if one can replace the Ambrosetti-Rabinowitz condition by more general superlinearity conditions, like the ones used in Papageorgiou and R\u adulescu \cite{29, 30}. Below we give simple examples of functions which satisfy hypotheses $H(f)$ (for the sake of simplicity, we drop the $z$-dependence):
	\begin{eqnarray*}
		&&f(x)=|x|^{r-2}x\ \mbox{for all}\ x\in\RR,\ \mbox{with}\ p<r<p^*,\\
		&&f(x)=\left\{\begin{array}{ll}
			|x|^{\eta-2}x&\mbox{if}\ |x|\leq 1\\
			2|x|^{r-2}x-|x|^{p-2}x&\mbox{if}\ 1<|x|
		\end{array}\right.\ \mbox{with}\ 1<p<\eta,\, r.
	\end{eqnarray*}
\end{remark}

One of our main results is that for all small $\lambda>0$, problem \eqref{eqP} admits extremal constant sign solutions, that is, there is a smallest positive solution $u_\lambda^*\in D_+$ and a biggest negative solution $v_\lambda^*\in -D_+$. These solutions are crucial in our proof on the existence of nodal (that is, sign changing) solutions (Section 4). To study the maps $\lambda\mapsto u_\lambda^*$ and $\lambda\mapsto v_\lambda^*$ and to prove the existence of nodal solutions, we will need to strengthen hypotheses $H(f)$ as follows.

\smallskip
$H(f)':$ $f:\Omega\times\RR\rightarrow\RR$ is a Carath\'eodory function such that $f(z,0)=0$ for almost all $z\in\Omega$, hypotheses $H(f)'(i),(ii),(iii)$ are the same as the corresponding hypotheses $H(f)(i),(ii),(iii)$, and \\
$(iv)$ for almost all $z\in\Omega,\ f(z,\cdot)$ is strictly increasing.
\begin{remark}
	The reason we impose this extra condition on $f(z,\cdot)$ is to be able to use the strong comparison principle in Proposition \ref{prop6}. The fact that the parameter $\lambda>0$ appears in the boundary and not in the reaction term, leads to stronger conditions on $f(z,\cdot)$.
\end{remark}

Finally in Section 5, where we deal with the semilinear problem (that is, $a(y)=y$ for all $y\in\RR^N$), in order to make use of tools from Morse theory (critical groups), we will need to introduce differentiability conditions on $f(z,\cdot)$. More precisely, the new hypotheses on $f(z,x)$ are:

\smallskip
$H(f)'':$ $f:\Omega\times\RR\rightarrow\RR$ is a measurable function such that for almost all $z\in\Omega$, $f(z,0)=0,f(z,\cdot)\in C^1(\RR)$ and
\begin{itemize}
	\item[(i)] $|f'_x(z,x)|\leq a(z)(1+|x|^{r-2})$ for almost all $z\in\Omega$ and all $x\in\RR$, with $a\in L^{\infty}(\Omega)$, $2<r<2^*$;
	\item[(ii)] if $F(z,x)=\int^x_0f(z,s)ds$, then there exist $\eta>2$ and $M>0$ such that
	$$0<\eta F(z,x)\leq f(z,x)x\ \mbox{and}\ f(z,x)x\leq c^*_1|x|^r-c^*_2|x|^2\ \mbox{for almost all}\ z\in\Omega\ \mbox{and all}\ |x|\geq M;$$
	\item[(iii)] $f'_x(z,0)=\lim\limits_{x\rightarrow 0}\frac{f(z,x)}{x}=0$ uniformly for almost all $z\in\Omega$;
	\item[(iv)] for every $\rho>0$, there exists $\hat{\xi}_{\rho}>0$ such that for almost all $z\in\Omega$ the function
	$$x\mapsto f(z,x)+\hat{\xi}_{\rho}x$$
	is nondecreasing on $[-\rho,\rho]$.
\end{itemize}
\begin{remark}
	Here hypothesis $H(f)''(iv)$ is much weaker than hypothesis $H(f)'(iv)$. The linearity of the differential operator leads to a more general strong comparison principle, which is a trivial consequence of the maximum principle.
\end{remark}

It is clear from the above hypotheses that in this paper we deal with subcritical reaction terms.

For the boundary function $\beta (z,x)$, we start with the following conditions.

\smallskip
$H(\beta):$ $\beta\in C(\partial\Omega\times\RR)\cap C^{0,\alpha}_{loc}(\partial\Omega\times\RR)$ for some $\alpha\in\left(0,1\right]$, $\beta(z,0)=0$ for all $z\in \partial\Omega$ and
\begin{itemize}
	\item[(i)] $c_{12}|x|^q\leq\beta(z,x)x$ for all $(z,x)\in \partial\Omega\times\RR$ and some $c_{12}>0$, with $q<\tau<p$ (see $H(a)(iv)$);
	\item[(ii)] $\lim\limits_{x\rightarrow\pm\infty}\frac{\beta(z,x)}{|x|^{p-1}x}=0$ uniformly for all $z\in \partial\Omega$;
	\item[(iii)] $\limsup\limits_{x\rightarrow 0}\frac{\beta(z,x)}{|x|^{q-2}x}\leq c_{13}$ uniformly for all $z\in \partial\Omega$, with $c_{13}>0$;
	\item[(iv)] if $B(z,x)=\int^x_0\beta(z,s)ds$ then $c_{14}|x|^q\leq\tau B(z,x)-\beta(z,x)x$ for all $(z,x)\in \partial\Omega\times\RR$ and some $c_9>0$ (see $H(a)(iv)$).
\end{itemize}
\begin{remark}
	The above hypotheses imply that
	\begin{equation}\label{eq19}
		|\beta(z,x)|\leq c_{15}|x|^{q-1}\ \mbox{for all}\ (z,x)\in \partial\Omega\times\RR\ \mbox{and some}\ c_{15}>0.
	\end{equation}
	
	So, the boundary term $\beta(z,\cdot)$ is strictly $(p-1)$-superlinear. The typical example of a function satisfying hypotheses $H(\beta)$ above is the following (for the sake of simplicity we again drop the $z$-dependence):
	$$\beta(x)=|x|^{q-2}x\ \mbox{for all}\ x\in\RR,\ \mbox{with}\ 1<q<\tau<p.$$
	
	Other possibilities are the functions
	\begin{eqnarray*}
		&&\beta(x)=|x|^{q-2}x+|x|^{\mu-2}x\ \mbox{for all}\ x\in\RR,\ \mbox{with}\ 1<q<\mu<\tau<p\\
		&&\beta(x)=\left\{\begin{array}{ll}
			|x|^{q-2}x&\mbox{if}\ |x|\leq 1\\
			2|x|^{\mu-2}x-|x|^{q-2}x&\mbox{if}\ |x|>1
		\end{array}\right.\ \mbox{with}\ 1<q<\mu<\tau<p,\\
		&&\hspace{7cm}q<\mu<2q,\ \mu<\frac{\tau(2q-\mu)}{q}.
	\end{eqnarray*}
\end{remark}

Later to deal with the semilinear problem we will need a stronger version of these conditions.

$H(\beta)':$ $\beta\in C(\partial\Omega\times\RR)\cap C^{0,\alpha}_{loc}(\partial\Omega\times\RR)$ with $\alpha\in(0,1)$, for all $z\in\partial\Omega$, $\beta(z,0)=0$, $\beta(z,\cdot)\in C^1(\RR\backslash\{0\})$ and
\begin{itemize}
	\item[(i)] $c_{50}|x|^q\leq\beta(z,x)x$ for all $(z,x)\in\partial\Omega\times\RR$, some $c_{50}>0$ and with $q\in(1,2)$;
	\item[(ii)] $\lim\limits_{x\rightarrow\pm\infty}\frac{\beta(z,x)}{x}=0$ uniformly for all $z\in\partial\Omega$;
	\item[(iii)] $\limsup\limits_{x\rightarrow 0}\frac{\beta(z,x)}{|x|^{q-2}x}\leq c_{51}$ uniformly for all $z\in\partial\Omega$, with $c_{51}>0$;
	\item[(iv)] if $B(z,x)=\int^x_0\beta(z,s)ds$, then $c_{52}|x|^q\leq 2B(z,x)-\beta(z,x)x$ for all $(z,x)\in\partial\Omega\times\RR$ and some $c_{52}>0$.
\end{itemize}

Now we state our main results.

\smallskip
{\bf Proposition A.} {\it
	If hypotheses $H(a),H(f),H(\beta)$ hold, then
	\begin{itemize}
		\item[(a)] for every $\lambda\in(0,\lambda_+)$ problem \eqref{eqP} admits two positive solutions
		$$u_0,\hat{u}\in D_+;$$
		\item[(b)] for every $\lambda\in(0,\lambda_-)$ problem \eqref{eqP} admits two negative solutions
		$$v_0,\hat{v}\in -D_+;$$
		\item[(c)] for every $\lambda\in(0,\lambda_0=\min\{\lambda_+,\lambda_-\})$ problem \eqref{eqP} admits four nontrivial constant sign solutions
		$$u_0,\hat{u}\in D_+\ \mbox{and}\ v_0,\hat{v}\in-D_+.$$
	\end{itemize}
}

\smallskip
{\bf Proposition B.} {\it
	If hypotheses $H(a),H(f),H(\beta)$ hold, then
	\begin{itemize}
		\item[(a)] for every $\lambda\in(0,\lambda_+)$ problem \eqref{eqP} has a smallest positive solution
		$$u^*_{\lambda}\in D_+;$$
		\item[(b)] for every $\lambda\in(0,\lambda_-)$ problem \eqref{eqP} has a biggest negative solution
		$$v^*_{\lambda}\in-D_+.$$
	\end{itemize}
}

\smallskip
{\bf Theorem C.} {\it
	If hypotheses $H(a),H(f)',H(\beta)$ hold, then there exists $\lambda_0>0$ such that for every $\lambda\in(0,\lambda_0)$ problem \eqref{eqP} has at least five nontrivial smooth solutions
	$$u_0,\hat{u}\in D_+,\ v_0,\hat{v}\in-D_+,\ y_0\in C^1(\overline{\Omega})\ \mbox{nodal}.$$
	Moreover, for every $\lambda\in(0,\lambda_0)$, problem \eqref{eqP} has extremal constant sign solutions
	$$u^*_{\lambda}\in D_+\ \mbox{and}\ v^*_{\lambda}\in-D_+$$
	such that $y_0\in[v^*_{\lambda},u^*_{\lambda}]\cap C^1(\overline{\Omega})$ and the map $\lambda\mapsto u^*_{\lambda}$ is
	\begin{itemize}
		\item strictly increasing (that is, $\mu<\lambda\Rightarrow u^*_{\lambda}-u^*_{\mu}\in {\rm int}\, \hat{C}_+$),
		\item left continuous from $(0,\lambda_0)$ into $C^1(\overline{\Omega})$,
	\end{itemize}
	while the map $\lambda\mapsto v^*_{\lambda}$ is
	\begin{itemize}
		\item	strictly decreasing (that is, $\mu<\lambda\Rightarrow v^*_{\mu}-v^*_{\lambda}\in {\rm int}\,\hat{C}_+$),
		\item right continuous.
	\end{itemize}
}

\smallskip
Finally, for the semilinear problem
\begin{equation}
	\left\{\begin{array}{ll}
		-\Delta u(z)=f(z,u(z))&\mbox{in}\ \Omega,\\
		\frac{\partial u}{\partial n}=\lambda\beta(z,u)&\mbox{on}\ \partial\Omega
	\end{array}\right\}\tag{$S_{\lambda}$}\label{eqS}
\end{equation}
we prove the following multiplicity result.

\smallskip
{\bf Theorem D.} {\it
	If hypotheses $H(f)'',\ H(\beta)'$ hold, then we can find $\lambda_0>0$ such that for every $\lambda\in(0,\lambda_0)$ problem \eqref{eqS} has at least six nontrivial smooth solutions
		\begin{eqnarray*}
				&&u_0,\hat{u}\in D_+,\ v_0,\hat{v}\in-D_+\\
				&&y_0\in C^1(\overline{\Omega})\ \mbox{nodal and}\ \hat{y}\in C^1(\overline{\Omega}).
		\end{eqnarray*}
}

\smallskip
Concluding this section, we point out that we use the word ``solution" instead of ``weak solution", since our solution has a pointwise a.e. interpretation (like the Carath\'eodory or strong solutions from the theory of ordinary differential equations). This pointwise interpretation of the solutions is convenient for the use of strong comparison principles (Proposition \ref{prop6}).

\section{Constant Sign Solutions}

In this section, we show that for small $\lambda>0$, problem \eqref{eqP} admits at least four nontrivial constant sign smooth solutions (two positive and two negative). We also establish the existence of extremal constant sign solutions $u^*_{\lambda}$, $v^*_{\lambda}$ and determine the monotonicity and continuity properties of the maps $\lambda\mapsto u^*_{\lambda}$ and $\lambda\mapsto v^*_{\lambda}$.

The energy (Euler) functional of problem \eqref{eqP} is  $\varphi_{\lambda}:W^{1,p}(\Omega)\rightarrow\RR$ ($\lambda>0$) and it is defined by
$$\varphi_{\lambda}(u)=\int_{\Omega}G(Du)dz-\int_{\Omega}F(z,u)dz-\lambda\int_{\partial\Omega}B(z,u)d\sigma\ \mbox{for all}\ u\in W^{1,p}(\Omega).$$

Evidently, $\varphi_{\lambda}\in C^1(W^{1,p}(\Omega))$.

Let $\tilde{c}_2\in(0,c^*_2)$ and consider the following truncation-perturbation of the reaction term $f(z,\cdot)$:
\begin{eqnarray}\label{eq20}
	&&\hat{f}_+(z,x)=\left\{\begin{array}{ll}
		0&\mbox{if}\ x\leq 0\\
		f(z,x)+\tilde{c}_2x^{p-1}&\mbox{if}\ 0<x
	\end{array}\right.\nonumber\\
&&\mbox{and} \\
	&&\hat{f}_-(z,x)=\left\{\begin{array}{ll}
		f(z,x)+\tilde{c}_2|x|^{p-2}x&\mbox{if}\ x<0\\
		0&\mbox{if}\ 0\leq x.\nonumber
	\end{array}\right.
\end{eqnarray}

Both are Carath\'eodory functions. We set $\hat{F}_{\pm}(z,x)=\int^x_0\hat{f}_{\pm}(z,s)ds$. In addition, we introduce the positive and negative truncations of the boundary term $\beta(z,\cdot)$:
\begin{eqnarray}\label{eq21}
	&&\beta_+(z,x)=\left\{\begin{array}{ll}
		0&\mbox{if}\ x\leq 0\\
		\beta(z,x)&\mbox{if}\ 0<x
	\end{array}\right. \nonumber\\
&&\mbox{and} \\
	&&\beta_-(z,x)=\left\{\begin{array}{ll}
		\beta(z,x)&\mbox{if}\ x<0\\
		0&\mbox{if}\ 0\leq x
	\end{array}\right.\ \mbox{for all}\ (z,x)\in \partial\Omega\times\RR.\nonumber
\end{eqnarray}

Clearly, $\beta_{\pm}\in C(\partial\Omega\times\RR)$. We set $B_{\pm}(z,x)=\int^x_0\beta_{\pm}(z,s)ds$ for all $(z,x)\in \partial\Omega\times\RR$. We consider the $C^1$-functionals $\hat{\varphi}^{\pm}_{\lambda}:W^{1,p}(\Omega)\rightarrow\RR$ $\lambda>0$, defined by
$$\hat{\varphi}^{\pm}_{\lambda}(u)=\int_{\Omega}G(Du)dz+\frac{\tilde{c}_2}{p}||u||^p_p-\int_{\Omega}\hat{F}_{\pm}(z,u)dz-\lambda\int_{\partial\Omega}B_{\pm}(z,u)d\sigma\ \mbox{for all}\ u\in W^{1,p}(\Omega).$$
\begin{prop}\label{prop8}
	If hypotheses $H(a),H(f),H(\beta)$ hold and $\lambda>0$, then the functionals $\hat{\varphi}^{\pm}_{\lambda}$ satisfy the $C$-function.
\end{prop}
\begin{proof}
	We give the proof for the functional $\hat{\varphi}^+_{\lambda}$, the proof for $\hat{\varphi}^-_{\lambda}$ being similar.
	
	So, we consider a sequence $\{u_n\}\subseteq W^{1,p}(\Omega)$ such that
	\begin{eqnarray}
		&&|\hat{\varphi}^+_{\lambda}(u_n)|\leq M_1\ \mbox{for some}\ M_1\ \mbox{and all}\ n\in\NN,\label{eq22}\\
		&&(1+||u_n||)(\hat{\varphi}^+_{\lambda})'(u_n)\rightarrow 0\ \mbox{in}\ W^{1,p}(\Omega)^*\ \mbox{as}\ n\rightarrow\infty\,.\label{eq23}
	\end{eqnarray}
	
	From (\ref{eq23}) we have
	\begin{eqnarray}\label{eq24}
		&&\left|\left\langle A(u_n),h\right\rangle+\int_{\Omega}\tilde{c}_2|u_n|^{p-2}u_nhdz-\int_{\Omega}\hat{f}_+(z,u_n)hdz-\lambda\int_{\partial\Omega}\beta_+(z,u_n)hd\sigma\right|\leq\frac{\epsilon_n||h||}{1+||u_n||}\\
		&&\mbox{for all}\ h\in W^{1,p}(\Omega)\ \mbox{with}\ \epsilon_n\rightarrow 0^+.\nonumber
	\end{eqnarray}
	
	In (\ref{eq24}) we choose $h=-u^-_n\in W^{1,p}(\Omega)$. Using (\ref{eq20}) and (\ref{eq21}), we obtain
	\begin{eqnarray}\label{eq25}
		&&\int_{\Omega}(a(Du_n),-Du^-_n)_{\RR^N}dz+\tilde{c}_2||u^-_n||^p_p\leq\epsilon_n\ \mbox{for all}\ n\in\NN,\nonumber\\
		&\Rightarrow&\frac{c_1}{p-1}||Du^-_n||^p_p+\tilde{c}_2||u^-_n||^p_p\leq\epsilon_n\ \mbox{for all}\ n\in\NN\ (\mbox{see Lemma \ref{lem2}}),\nonumber\\
		&\Rightarrow&u^-_n\rightarrow 0\ \mbox{in}\ W^{1,p}(\Omega).
	\end{eqnarray}
	
	Using (\ref{eq19}), (\ref{eq21}), (\ref{eq20}), (\ref{eq25}) and hypothesis $H(f)(i)$, we have
	\begin{eqnarray}\label{eq26}
		&&\int_{\Omega}pG(Du^+_n)dz-\int_{\Omega}pF(z,u^+_n)dz-\lambda\int_{\partial\Omega}pB(z,u^+_n)d\sigma\leq M_2\\
		&&\mbox{for some}\ M_2>0,\ \mbox{all}\ n\in\NN.\nonumber
	\end{eqnarray}
	
	In (\ref{eq24}) we choose $h=u^+_n\in W^{1,p}(\Omega)$. Then
	\begin{eqnarray}\label{eq27}
		 &&-\int_{\Omega}(a(Du^+_n),Du^+_n)_{\RR^N}dz+\int_{\Omega}f(z,u^+_n)u^+_ndz+\lambda\int_{\partial\Omega}\beta(z,u^+_n)u^+_nd\sigma\leq\epsilon_n\\
		&&\mbox{for all}\ n\in\NN.\nonumber
	\end{eqnarray}
	
	Adding (\ref{eq26}) and (\ref{eq27}), we obtain
	\begin{eqnarray}\label{eq28}
		&&\int_{\Omega}[pG(Du^+_n)-(a(Du^+_n),Du^+_n)_{\RR^N}]dz+\int_{\Omega}[f(z,u^+_n)u^+_n-pF(z,u^+_n)]dz\nonumber\\
		&&\leq M_3+\lambda\int_{\partial\Omega}[pB(z,u^+_n)-\beta(z,u^+_n)u^+_n]d\sigma\ \mbox{for some}\ M_3>0\ \mbox{and all}\ n\in\NN\nonumber\\
		&\Rightarrow&\int_{\Omega}[f(z,u^+_n)u^+_n-pF(z,u^+_n)]dz\leq c_{16}(1+||u^+_n||^q)\ \mbox{for some}\ c_{16}>0\ \mbox{and all}\ n\in\NN\nonumber\\
		&&(\mbox{see hypothesis}\ H(a)(iv)\ \mbox{and (\ref{eq19})})\nonumber\\
		&\Rightarrow&\int_{\Omega}[f(z,u^+_n)u^+_n-\eta F(z,u^+_n)]dz+(\eta-p)\int_{\Omega}F(z,u^+_n)dz\leq c_{16}(1+||u^+_n||^q)\nonumber\\
		&&\mbox{for all}\ n\in\NN,\nonumber\\
		&\Rightarrow&(\eta-p)\int_{\Omega}F(z,u^+_n)dz\leq c_{17}(1+||u^+_n||^q)\ \mbox{for some}\ c_{17}>0\ \mbox{and all}\ n\in\NN\\
		&&(\mbox{see hypotheses}\ H(f)(i),(ii)).\nonumber
	\end{eqnarray}
	
	From (\ref{eq18}) and hypothesis $H(f)(i)$, we see that
	\begin{equation}\label{eq29}
		c_{11}|x|^{\eta}-c_{18}\leq F(z,x)\ \mbox{for almost all}\ z\in\Omega,\ \mbox{all}\ x\in\RR,\ \mbox{and some}\ c_{18}>0.
	\end{equation}
	
	Using (\ref{eq29}) in (\ref{eq28}) and recalling that $\eta>p$, we obtain
	\begin{eqnarray}\label{eq30}
		&||u^+_n||^{\eta}_{\eta}&\leq c_{19}(1+||u^+_n||^q)\ \mbox{for some}\ c_{19}>0\ \mbox{and all}\ n\in\NN,\nonumber\\
		\Rightarrow&||u^+_n||^p_{\eta}&\leq c^{p/\eta}_{19}(1+||u^+_n||^q)^{p/\eta}\nonumber\\
		&&\leq c_{20}(1+||u^+_n||^{pq/\eta})\ \mbox{for}\ c_{20}=c^{p/\eta}_{19}\ \mbox{and all}\ n\in\NN\ (\mbox{note that}\ \frac{p}{\eta}\in(0,1)).
	\end{eqnarray}
	
It follows	from (\ref{eq22}) and (\ref{eq25})  that for all $n\in\NN$
	\begin{equation}\label{eq31}
		\int_{\Omega}\eta G(Du^+_n)dz-\int_{\Omega}\eta F(z,u^+_n)dz-\lambda\int_{\partial\Omega}\eta B(z,u^+_n)d\sigma\leq M_4,\ \mbox{for some}\ M_4>0.
	\end{equation}
	
	Adding (\ref{eq27}) and (\ref{eq31}), we have
	\begin{eqnarray}\label{eq32}
		&&\int_{\Omega}[\eta G(Du^+_n)-(a(Du^+_n),Du^+_n)_{\RR^N}]dz+\int_{\Omega}[f(z,u^+_n)u^+_n-\eta F(z,u^+_n)]dz\nonumber\\
		&&\leq M_5+\lambda\int_{\partial\Omega}[\eta B(z,u^+_n)-\beta(z,u^+_n)u^+_n]d\sigma\ \mbox{for some}\ M_5>0\ \mbox{and all}\ n\in\NN.
	\end{eqnarray}
	
	Note that
	\begin{eqnarray}\label{eq33}
		&&\int_{\Omega}[\eta G(Du^+_n)-(a(Du^+_n),Du^+_n)_{\RR^N}]dz\nonumber\\
		&&=(\eta-p)\int_{\Omega}G(Du^+_n)dz+\int_{\Omega}[pG(Du^+_n)-(a(Du^+_n),Du^+_n)_{\RR^N}]dz\nonumber\\
		&&\geq\frac{(\eta-p)c_1}{p(p-1)}||Du^+_n||^p_p-\bar{c}|\Omega|_n\ (\mbox{see Corollary \ref{cor3} and hypothesis}\ H(a)(iv)).
	\end{eqnarray}
	
	Also, hypotheses $H(f)(i),(ii)$ imply that
	\begin{equation}\label{eq34}
		-c_{21}\leq\int_{\Omega}[f(z,u^+_n)u^+_n-\eta F(z,u^+_n)]dz\ \mbox{for some}\ c_{21}>0\ \mbox{and all}\ n\in\NN.
	\end{equation}
	
	Moreover, from (\ref{eq19}) we have
	\begin{eqnarray}\label{eq35}
		&&\int_{\Omega}[\eta B(z,u^+_n)-\beta(z,u^+_n)u^+_n]d\sigma\leq c_{22}||u^+_n||^q\ \mbox{for some}\ c_{22}>0\ \mbox{and all}\ n\in\NN.
	\end{eqnarray}
	
	Returning to (\ref{eq32}) and using (\ref{eq33}), (\ref{eq34}), (\ref{eq35}), we obtain
	\begin{equation}\label{eq36}
		||Du^+_n||^p_p\leq c_{23}(1+||u^+_n||^q)\ \mbox{for some}\ c_{23}>0\ \mbox{and all}\ n\in\NN.
	\end{equation}
	
	It follows from (\ref{eq30}) and (\ref{eq36}) that
	\begin{equation}\label{eq37}
		||Du^+_n||^p_p+||u^+_n||^p_{\eta}\leq c_{24}(1+||u^+_n||^q)\ \mbox{for some}\ c_{24}>0\ \mbox{and all}\ n\in\NN.
	\end{equation}
	
	We can always assume that $\eta\leq p^*$ (see hypotheses $H(f)(i),(ii)$). We know that
	$$u\mapsto ||Du||_p+||u||_{\eta}$$
	is an equivalent norm on the Sobolev space $W^{1,p}(\Omega)$ (see Gasinski and Papageorgiou \cite[p. 227]{13}). Therefore from (\ref{eq37}) we can infer that
	\begin{eqnarray*}
		&&||u^+_n||^p\leq c_{25}(1+||u^+_n||^q)\ \mbox{for some}\ c_{20}>0\ \mbox{and all}\ n\in\NN,\\
		&\Rightarrow&\{u^+_n\}_{n\geq 1}\subseteq W^{1,p}(\Omega)\ \mbox{is bounded}\ (\mbox{recall that}\ q<p).
	\end{eqnarray*}
	
	This together with (\ref{eq25}) imply that $\{u_n\}_{n\geq 1}\subseteq W^{1,p}(\Omega)$ is bounded and so we may assume that
	\begin{eqnarray}\label{eq38}
		&&u_n\stackrel{w}{\rightarrow}u\ \mbox{in}\ W^{1,p}(\Omega)\ \mbox{and}\ u_n\rightarrow u\ \mbox{in}\ L^r(\Omega)\ \mbox{and in}\ L^q(\partial\Omega)\ (\mbox{recall that}\ r<p^*).
	\end{eqnarray}
	
	In (\ref{eq24}) we choose $h=u_n-u\in W^{1,p}(\Omega)$, pass to the limit as $n\rightarrow\infty$, and use (\ref{eq38}). Then
	\begin{eqnarray*}
		&&\lim\limits_{n\rightarrow\infty}\left\langle A(u_n),u_n-u\right\rangle=0,\\
		&\Rightarrow&u_n\rightarrow u\ \mbox{in}\ W^{1,p}(\Omega)\ (\mbox{see (\ref{eq38}) and Proposition \ref{prop4}}),\\
		&\Rightarrow&\hat{\varphi}^+_{\lambda}\ \mbox{satisifes the C-condition}.
	\end{eqnarray*}
	
	Similarly for the functional $\hat{\varphi}^-_{\lambda}$.
\end{proof}

In a similar fashion, we prove the next property.
\begin{prop}\label{prop9}
	If hypotheses $H(a),H(f),H(\beta)$ hold and $\lambda>0$, then the functional $\varphi_{\lambda}$ satisfies the $C$-condition.
\end{prop}

Next, we provide the mountain pass geometry for the functional $\hat{\varphi}^{\pm}_{\lambda}$.
\begin{prop}\label{prop10}
	If hypotheses $H(a),H(f),H(\beta)$ hold, then there exists $\lambda_{\pm}>0$ such that for every $\lambda\in(0,\lambda_{\pm})$ we can find $\rho^{\pm}_{\lambda}$ for which we have
	 $$\inf[\hat{\varphi}^{\pm}_{\lambda}(u):||u||=\rho^{\pm}_{\lambda}]=\hat{m}^{\pm}_{\lambda}>0=\hat{\varphi}^{\pm}_{\lambda}(0).$$
\end{prop}
\begin{proof}
	We again present the proof only for $\hat{\varphi}^+_{\lambda}$, since the proof for $\hat{\varphi}^-_{\lambda}$ is similar.
	
	Hypotheses $H(f)$ imply that given $\epsilon>0$, we can find $\delta>0$ and $c_{26}>0$ such that
	\begin{eqnarray}
		&&|F(z,x)|\leq\frac{\epsilon}{p}|x|^p\ \mbox{for almost all}\ z\in\Omega\ \mbox{and all}\ |x|\leq\delta\label{eq39}\\
		&&F(z,x)\leq c_{26}|x|^r-\frac{c^*_2}{p}|x|^p\ \mbox{for almost all}\ z\in\Omega\ \mbox{and all}\ |x|>\delta\,.\label{eq40}
	\end{eqnarray}
	
	Similarly, hypotheses $H(\beta)(i),(ii),(iii)$, imply that given $\epsilon>0$, we can find $c_{27}>0$ such that
	\begin{equation}\label{eq41}
		B(z,x)\leq\epsilon|x|^p+c_{27}|x|^q\ \mbox{for all}\ (z,x)\in\partial\Omega\times\RR.
	\end{equation}
	
	Then for all $u\in W^{1,p}(\Omega)$, we have
	\begin{eqnarray}\label{eq42}
		 &&\hat{\varphi}^+_{\lambda}(u)\geq\frac{c_1}{p(p-1)}||Du||^p_p+\frac{\tilde{c}_2}{p}||u||^p_p-\int_{\Omega}F(z,u^+)dz-\lambda\int_{\partial\Omega}B(z,u^+)d\sigma-\frac{\tilde{c}_2}{p}||u^+||^p_p\\
		&&(\mbox{see Corollary \ref{cor3} and (\ref{eq20})}).\nonumber
	\end{eqnarray}
	
	We have
	\begin{eqnarray*}
		\int_{\Omega}F(z,u^+)dz&=&\int_{\{u^+>\delta\}}F(z,u^+)dz+\int_{\{u^+\leq\delta\}}F(z,u^+)dz\\
		&\leq&c_{26}||u||^r-\frac{c^*_2}{p}\int_{\{u^+>\delta\}}(u^+)^pdz+\frac{\epsilon}{p}||u||^p_p\ (\mbox{see (\ref{eq39}) and (\ref{eq40})}).
	\end{eqnarray*}
	
	Also, we have
	$$\lambda\int_{\partial\Omega}B(z,u^+)d\sigma\leq\lambda\epsilon c_{28}||u||^p+\lambda c_{29}||u||^q\ \mbox{for some}\ c_{28},c_{29}>0\ \mbox{(see (\ref{eq41}))}.$$
	
	Using these two estimates and choosing small $\epsilon>0$, we obtain
	$$\hat{\varphi}^+_{\lambda}(u)\geq c_{30}||u||^p-c_{26}||u||^r-\lambda c_{29}||u||^q+\frac{c^*_2}{p}\int_{\{u^+>\delta\}}(u^+)^pdz-\frac{\tilde{c}_2}{p}||u^+||^p_p.$$
	
	Note that
	$$\int_{\{u^+>\delta\}}(u^+)^pdz\rightarrow||u^+||^p_p\ \mbox{as}\ \delta\rightarrow 0^+.$$
	
	So, given $\vartheta>0$, we can find $\delta_0>0$ such that
	$$\frac{c^*_2}{p}\int_{\{u^+>\delta\}}(u^+)^pdz\geq\frac{c^*_2}{p}(1-\vartheta)||u^+||^p_p\ \mbox{for all}\ 0<\delta\leq\delta_0.$$
	
	Then we have
	$$\hat{\varphi}^+_{\lambda}(u)\geq c_{30}||u||^p-c_{26}||u||^r-\lambda c_{29}||u||^q+\frac{1}{p}[c^*_2(1-\vartheta)-\tilde{c}_2]||u^+||^p_p.$$
	
	Since $c^*_2>\tilde{c}_2$ we choose small $\vartheta>0$ small such that $c^*_2(1-\vartheta)\geq\tilde{c}_2$. Then
	\begin{equation}\label{eq43}
		\hat{\varphi}^+_{\lambda}(u)\geq[c_{30}-(c_{26}||u||^{r-p}+\lambda c_{29}||u||^{q-p})]||u||^p.
	\end{equation}
	
	Consider the function
	$$\Im_{\lambda}(t)=c_{26}t^{r-q}+\lambda c_{29}t^{q-p}\ \mbox{for all}\ t>0.$$
	
	Since $q<p<r$, we see that $\Im_{\lambda}(t)\rightarrow+\infty$ as $t\rightarrow 0^+$ and $t\rightarrow+\infty$. So, we can find $t_0>0$ such that
	\begin{eqnarray*}
		&&\Im_{\lambda}(t_0)=\inf\limits_{\RR_+}\Im_{\lambda},\\
		&\Rightarrow&\Im'_{\lambda}(t_0)=0,\\
		&\Rightarrow&t_0=\left[\frac{(p-q)\lambda c_{29}}{(r-q)c_{26}}\right]^{\frac{1}{r-q}}.
	\end{eqnarray*}
	
	Then $\Im_{\lambda}(t_0)\rightarrow 0$ as $\lambda\rightarrow 0^+$ and so we can find small $\lambda_+>0$ such that $\Im_{\lambda}(t_0)<c_{30}$ for all $\lambda\in (0,\lambda_+)$. From (\ref{eq43}) we see that
	$$\hat{\varphi}^+_{\lambda}(u)\geq\hat{m}^+_{\lambda}>0=\hat{\varphi}^+_{\lambda}(0)\ \mbox{for all}\ u\in W^{1,p}(\Omega)\ \mbox{with}\ ||u||=\rho^+_{\lambda}=t_0(\lambda).$$
	
	Similarly, we show that there exists $\lambda_->0$ such that for every $\lambda\in(0,\lambda_-)$ we can find $\rho^-_{\lambda}>0$ for which we have
	$$\hat{\varphi}^-_{\lambda}(u)\geq\hat{m}^-_{\lambda}>0=\hat{\varphi}^-_{\lambda}(0)\ \mbox{for all}\ u\in W^{1,p}(\Omega)\ \mbox{with}\ ||u||=\rho^-_{\lambda}=t_0(\lambda).$$
\end{proof}

It is immediate from hypothesis $H(f)(ii)$ (see also (\ref{eq18})) that:
\begin{prop}\label{prop11}
	If hypotheses $H(a),H(f),H(\beta)$ hold, $u\in D_+$, and $\lambda>0$, then $\hat{\varphi}^+_{\lambda}(tu)\rightarrow-\infty$ as $t\rightarrow+\infty$.
\end{prop}

Now, we are ready to produce constant sign solutions.
\begin{prop}\label{prop12}
	If hypotheses $H(a),H(f),H(\beta)$ hold, then
	\begin{itemize}
		\item[(a)] for every $\lambda\in(0,\lambda_+)$ problem \eqref{eqP} admits two positive solutions
		$$u_0,\hat{u}\in D_+;$$
		\item[(b)] for every $\lambda\in(0,\lambda_-)$ problem \eqref{eqP} admits two negative solutions
		$$v_0,\hat{v}\in -D_+;$$
		\item[(c)] for every $\lambda\in(0,\lambda_0=\min\{\lambda_+,\lambda_-\})$ problem \eqref{eqP} admits four nontrivial constant sign solutions
		$$u_0,\hat{u}\in D_+\ \mbox{and}\ v_0,\hat{v}\in-D_+.$$
	\end{itemize}
\end{prop}
\begin{proof}
	\textit{(a)} Let $\lambda\in(0,\lambda_+)$ and let $\rho^+_{\lambda}$ be as postulated by Proposition \ref{prop10}. We consider the set
	$$\bar{B}_{\rho^+_{\lambda}}=\{u\in W^{1,p}(\Omega):||u||\leq\rho^+_{\lambda}\}.$$
	
	This set is weakly compact in $W^{1,p}(\Omega)$. Moreover, using the Sobolev embedding theorem and the compactness of the trace map, we see that $\hat{\varphi}^+_{\lambda}$ is sequentially weakly lower semicontinuous. So, by the Weierstrass theorem, we can find $u_0\in W^{1,p}(\Omega)$ such that
	\begin{equation}\label{eq44}
		\hat{\varphi}^+_{\lambda}(u_0)=\inf[\hat{\varphi}^+_{\lambda}(u):u\in\bar{B}_{\rho^+_{\lambda}}(u_0)].
	\end{equation}
	
	Hypotheses $H(a)(iv),H(f)(iii)$ imply that we can find $\delta\in(0,1)$ such that
	\begin{eqnarray}
		&&G(y)\leq c_{31}|y|^{\tau}\ \mbox{for all}\ |y|\leq\delta\ \mbox{with}\ c_{31}>0,\label{eq45}\\
		&&|F(z,x)|\leq|x|^p\ \mbox{for almost all}\ z\in\Omega\ \mbox{and all}\ |x|\leq\delta\,.\label{eq46}
	\end{eqnarray}
	
	Then for $u\in C_+\backslash\{0\}$ with $||u||_{C^1(\overline{\Omega})}\leq\delta$, we have
	\begin{eqnarray}\label{eq47}
		&&\hat{\varphi}^+_{\lambda}(u)\leq c_{32}\delta^{\tau}-\delta^p|\Omega|_N-\frac{\lambda c_{12}}{q}\delta^q|\Omega|_N\ \mbox{for some}\ c_{32}>0\\
		&&(\mbox{see (\ref{eq45}), (\ref{eq46}) and hypothesis}\ H(\beta)(i)).\nonumber
	\end{eqnarray}
	
	Since $q<\tau<p$, by taking $\delta\in(0,1)$ even smaller if necessary, we infer from (\ref{eq47}) that
	\begin{equation}\label{eq48}
		\hat{\varphi}^+_{\lambda}(u)<0\ \mbox{and}\ ||u||\leq\rho^+_{\lambda}.
	\end{equation}
	
	It follows from (\ref{eq44}) and (\ref{eq48}) that
	\begin{eqnarray}
		&&\hat{\varphi}^+_{\lambda}(u_0)<0=\hat{\varphi}^+_{\lambda}(0),\label{eq49}\\
		&\Rightarrow&u_0\neq 0\ \mbox{and}\ ||u_0||<\rho^+_{\lambda}\ (\mbox{see Proposition \ref{prop10}}).\label{eq50}
	\end{eqnarray}
	
	From (\ref{eq44}) and (\ref{eq50}) we have
	\begin{eqnarray}\label{eq51}
		&&(\hat{\varphi}^+_{\lambda})'(u_0)=0\nonumber\\
		&\Rightarrow&\left\langle A(u_0),h\right\rangle+\int_{\Omega}c^*_2|u_0|^{p-2}u_0hdz=\int_{\Omega}\hat{f}_+(z,u_0)hdz+\lambda\int_{\partial\Omega}\beta(z,u^+_0)hd\sigma\\
		&&\mbox{for all}\ h\in W^{1,p}(\Omega).\nonumber
	\end{eqnarray}
	
	In (\ref{eq51}) we choose $h=-u^-_0\in W^{1,p}(\Omega)$. Then using Lemma \ref{lem2} and (\ref{eq20}), (\ref{eq21}), we obtain
	\begin{eqnarray*}
		&&\frac{c_1}{p-1}||Du^-_0||^p_p+c^*_2||u^-||^p_p\leq 0,\\
		&\Rightarrow&u_0\geq 0,u_0\neq 0.
	\end{eqnarray*}
	
	Hence equation (\ref{eq51}) becomes
	\begin{eqnarray}\label{eq52}
		&&\left\langle A(u_0),h\right\rangle=\int_{\Omega}f(z,u_0)hdz+\lambda\int_{\partial\Omega}\beta(z,u_0)hd\sigma\ \mbox{for all}\ h\in W^{1,p}(\Omega),\nonumber\\
		&\Rightarrow&-{\rm div}\, a(Du_0(z))=f(z,u_0(z))\ \mbox{for almost all}\ z\in\Omega,\ \frac{\partial u}{\partial n_a}=\lambda\beta(z,u_0)\ \mbox{on}\ \partial\Omega\\
		&&(\mbox{see Papageorgiou and R\u adulescu \cite{28}}).\nonumber
	\end{eqnarray}
	
	From Hu and Papageorgiou \cite{19} and Papageorgiou and R\u adulescu \cite{30}, we have
	$$u_0\in L^{\infty}(\Omega).$$
	
	Then invoking the nonlinear regularity theory of Lieberman \cite[p. 320]{21}, we can infer that
	$$u_0\in C_+\backslash\{0\}.$$
	
	Hypotheses $H(f)(i),(iii)$ imply that given $\rho>0$, we can find $\hat{\xi}_{\rho}>0$ such that
	$$f(z,x)x+\hat{\xi}_{\rho}|x|^p\geq 0\ \mbox{for almost all}\ z\in\Omega\ \mbox{and all}\ |x|\leq \rho\,.$$
	
	If $\rho=||u_0||_{\infty}$, from (\ref{eq52}) we have
	\begin{equation}\label{eq53}
		{\rm div}\, a(Du_0(z))\leq\hat{\xi}_{\rho}u_0(z)^{p-1}\ \mbox{for almost all}\ z\in\Omega\,.
	\end{equation}
	
	Let $\mu(t)=a_0(t)t,\ t>0$. Then
	$$\mu'(t)t=a'_0(t)t^2+a_0(t)t.$$
	
	Performing integration by parts, we obtain
	\begin{eqnarray*}
		\int^t_0\mu'(s)sds&=&\mu(t)t-\int^t_0\mu(s)ds\\
		&=&a_0(t)t^2-G_0(t)\\
		&\geq&c_4t^p\ (\mbox{see hypothesis}\ H(a)(iv)).
	\end{eqnarray*}
	
	We set $H(t)=a_0(t)t^2-G_0(t),H_0(t)=c_4t^p$ for all $t\geq 0$. Let $\delta\in(0,1)$ and $s>0$. We introduce the sets
	$$C_1=\{t\in(0,1):H(t)\geq s\}\ \mbox{and}\ C_2=\{t\in(0,1):H_0(t)\geq s\}.$$
	
	Then $C_2\subseteq C_1$ and so $\inf C_1\leq \inf C_2$. Hence
	\begin{eqnarray*}
		&&H^{-1}(s)\leq H^{-1}_0(s)\ (\mbox{see Gasinski and Papageorgiou \cite[Proposition 1.55]{16}})\\
		&\Rightarrow&\int^{\delta}_0\frac{1}{H^{-1}(\frac{\hat{\xi}}{p}s^p)}ds
\geq\int^{\delta}_0\frac{1}{H^{-1}_0(\frac{\hat{\xi}}{p}s^p)}ds=
\frac{\hat{\xi}_{\rho}}{p}\int^{\delta}_0\frac{ds}{s}=+\infty\,.
	\end{eqnarray*}
	
	Because of (\ref{eq53}) we can apply the nonlinear strong maximum principle of Pucci and Serrin \cite[p. 111]{33}, from which we obtain
	$$0<u(z)\ \mbox{for all}\ z\in\Omega.$$
	
	Then the boundary point theorem of Pucci and Serrin \cite[p. 120]{33}, implies that
	$$u_0\in D_+.$$
	
	Next, note that Propositions \ref{prop8}, \ref{prop10} and \ref{prop11} permit the use of Theorem \ref{th1} (the mountain pass theorem) on the functional $\hat{\varphi}^+_{\lambda}$ $(\lambda\in(0,\lambda_+))$. So, we can find $\hat{u}\in W^{1,p}(\Omega)$ such that
	\begin{equation}\label{eq54}
		\hat{u}\in K_{\hat{\varphi}^+_{\lambda}}\ \mbox{and}\ \hat{m}^+_{\lambda}\leq \hat{\varphi}^+_{\lambda}(\hat{u}).
	\end{equation}
	
It follows	from (\ref{eq54})  that
	$$\hat{u}\notin \{0,u_0\}\ (\mbox{see Proposition \ref{prop10} and (\ref{eq49})}).$$
	
	As before we can easily check that
	$$K_{\hat{\varphi}^+_{\lambda}}\backslash\{0\}\subseteq D_+
		\Rightarrow\hat{u}\in D_+\ (\mbox{see (\ref{eq54})}).
	$$
	
	\textit{(b)} Similarly, working this time with the functional $\hat{\varphi}^-_{\lambda}\ (\lambda\in(0,\lambda_-))$, we produce two negative solutions
	$$v_0,\hat{v}\in-D_+.$$
	
	\textit{(c)} This part follows from \textit{(a)} and \textit{(b)} above.
\end{proof}

In fact, we can show the existence of extremal constant sign solutions, that is, we will show the following:
\begin{itemize}
	\item for every $\lambda\in(0,\lambda_+)$, problem \eqref{eqP} has a smallest positive solution $u^*_{\lambda}\in D_+$;
	\item for every $\lambda\in(0,\lambda_-)$, problem \eqref{eqP} has a biggest negative solution $v^*_{\lambda}\in-D_+$.
\end{itemize}

To this end, note that hypotheses $H(f)$ imply that
\begin{equation}\label{eq55}
	f(z,x)x\geq-c_{33}|x|^p\ \mbox{for almost all}\ z\in\Omega\ \mbox{and all}\ x\in\RR,\ \mbox{with}\ c_{33}>0.
\end{equation}

This unilateral growth estimate on the reaction term $f(z,\cdot)$ and hypothesis $H(\beta)(i)$, lead to the following auxiliary nonlinear boundary value problem:
\begin{equation}
	\left\{\begin{array}{ll}
		-{\rm div}\, a(Du(z))+c_{33}|u(z)|^{p-2}u(z)=0&\mbox{in}\ \Omega,\\
		\frac{\partial u}{\partial n_a}=\lambda c_{12}|u|^{q-2}u&\mbox{on}\ \partial\Omega
	\end{array}\right\}\tag{$Au_{\lambda}$}\label{eqA}
\end{equation}
\begin{prop}\label{prop13}
	If hypotheses $H(a)$ hold and $\lambda>0$ then problem \eqref{eqA} has a unique positive solution $\bar{u}_{\lambda}\in D_+$ and a unique negative solution $\bar{v}_{\lambda}\in-D_+$.
\end{prop}
\begin{proof}
	First, we establish the existence of a positive solution.
	
	So we consider the $C^1$-functional $\psi^+_{\lambda}:W^{1,p}(\Omega)\rightarrow\RR$ defined by
	\begin{eqnarray*}
		&&\psi^+_{\lambda}(u)=\int_{\Omega}G(Du)dz+\frac{1}{p}||u||^p_p+(c_{33}-1)||u^+||^p_p-\frac{\lambda c_{12}}{q}\int_{\partial\Omega}(u^+)^qd\sigma\\
		&&\mbox{for all}\ u\in W^{1,p}(\Omega).
	\end{eqnarray*}
	
	Evidently, we can always assume that $c_{33}>1$ (see (\ref{eq55})). Then we have
	\begin{eqnarray*}
		 &&\psi^+_{\lambda}(u)\geq\frac{c_1}{p(p-1)}||Du^+||^p_p+c_{34}||u^+||^p_p+\frac{c_1}{p(p-1)}||Du^-||^p_p+\frac{1}{p}||u^-||^p_p-\lambda c_{35}||u||^q\\
		&&\mbox{for some}\ c_{34},c_{35}>0\ \mbox{(see Corollary \ref{cor3})},\\
		&&\geq c_{36}||u||^p-\lambda c_{35}||u||^q\ \mbox{for some}\ c_{36}>0,\\
		&\Rightarrow&\psi^+_{\lambda}\ \mbox{is coercive}\ (\mbox{recall that}\ q<p).
	\end{eqnarray*}
	
	Moreover, the Sobolev embedding theorem and the compactness of the trace map, imply that $\psi^+_{\lambda}$ is sequentially weakly lower semicontinuous. So, we can find $\bar{u}_{\lambda}\in W^{1,p}(\Omega)$ such that
	\begin{equation}\label{eq56}
		\psi^+_{\lambda}(\bar{u}_{\lambda})=\inf[\psi^+_{\lambda}(u):u\in W^{1,p}(\Omega)].
	\end{equation}
	
	As in the proof of Proposition \ref{prop12}, exploiting the fact that $q<\tau<p$, we show that
	\begin{eqnarray*}
		&&\psi^+_{\lambda}(\bar{u}_{\lambda})<0=\psi^+_{\lambda}(0),\\
		&\Rightarrow&\bar{u}_{\lambda}\neq 0.
	\end{eqnarray*}
	
	From (\ref{eq56}) we have
	\begin{eqnarray}\label{eq57}
		&&(\psi^+_{\lambda})'(\bar{u}_{\lambda})=0,\nonumber\\
		&\Rightarrow&\left\langle A(\bar{u}_{\lambda}),h\right\rangle+\int_{\Omega}|\bar{u}_{\lambda}|^{p-2}\bar{u}_{\lambda}hdz+(c_{33}-1)\int_{\Omega}(\bar{u}^+_{\lambda})^{p-1}hdz=\lambda\int_{\partial\Omega}c_{12}(\bar{u}^+_{\lambda})^{q-1}hd\sigma\\
		&&\mbox{for all}\ h\in W^{1,p}(\Omega).\nonumber
	\end{eqnarray}
	
	In (\ref{eq57}) we choose $h=-\bar{u}^-_{\lambda}\in W^{1,p}(\Omega)$. Then
	\begin{eqnarray*}
		&&\frac{c_1}{p-1}||D\bar{u}^-_{\lambda}||^p_p+||\bar{u}^-_{\lambda}||^p_p\leq 0\ (\mbox{see Lemma \ref{lem2}}),\\
		&\Rightarrow&\bar{u}_{\lambda}\geq 0,\bar{u}_{\lambda}\neq 0.
	\end{eqnarray*}
	
	Then equation (\ref{eq57}) becomes
	\begin{eqnarray*}
		&&\left\langle A(\bar{u}_{\lambda}),h\right\rangle+\int_{\Omega}c_{33}\bar{u}^{p-1}_{\lambda}hdz-\lambda\int_{\partial\Omega}c_{12}\bar{u}^{q-1}_{\lambda}hd\sigma=0\ \mbox{for all}\ h\in W^{1,p}(\Omega),\\
		&\Rightarrow&-{\rm div}\, a(D\bar{u}_{\lambda}(z))+c_{33}\bar{u}_{\lambda}(z)^{p-1}=0\ \mbox{for almost all}\ z\in\Omega,\\
		&&\frac{\partial\bar{u}_{\lambda}}{\partial n_a}=\lambda c_{12}\bar{u}^{q-1}_{\lambda}\ \mbox{on}\ \partial\Omega.
	\end{eqnarray*}
	
	As before, the nonlinear regularity theory (see Lieberman \cite{21}) and the nonlinear maximum principle (see \cite{33}), imply that
	$$\bar{u}_{\lambda}\in D_+.$$
	
	Next, we show the uniqueness of this positive solution. To this end, we introduce the integral functional $j:L^1(\Omega)\rightarrow\bar{\RR}=\RR\cup\{+\infty\}$ defined  by
	\begin{eqnarray}\label{eq58}
		&&j(u)=\left\{\begin{array}{ll}
			\int_{\Omega}G(Du^{1/\tau})dz-\frac{c_{12}\tau}{q}\int_{\partial\Omega}u^{q/\tau}d\sigma&\mbox{if}\ u\geq 0,u^{1/\tau}\in W^{1,p}(\Omega)\\
			+\infty&\mbox{otherwise}.
		\end{array}\right.
	\end{eqnarray}
	
	Let ${\rm dom}\, j=\{u\in L^1(\Omega):j(u)<+\infty\}$ (the effective domain of $j$) and let $u_1,u_2\in {\rm dom}\, j$. We set
	$$u=((1-t)u_1+tu_2)^{1/\tau}\ \mbox{for}\ t\in[0,1].$$
	
	From Lemma 1 of Diaz and Saa \cite{9}, we have
	$$|Du(z)|\leq\left[(1-t)|Du_1(z)^{1/\tau}|^{\tau}+t|Du_2(z)^{1/\tau}|^{\tau}\right]^{1/\tau}\ \mbox{for almost all}\ z\in\Omega.$$
	
	Then
	\begin{eqnarray*}
		G_0(|Du(z)|)&\leq&G_0\left([(1-t)|Du_1(z)^{1/\tau}|^{\tau}+t|Du_2(z)^{1/\tau}|^{\tau}]^{1/\tau}\right)\\
		&&(\mbox{recall that}\ G_0(\cdot)\ \mbox{is increasing})\\
		&\leq&(1-t)G_0(|Du_1(z)^{1/\tau}|)+tG_0(|Du_2(z)^{1/\tau}|)\ (\mbox{see hypothesis}\ H(a)(iv))\\
		\Rightarrow G(Du(z))&\leq&(1-t)G(Du_1(z)^{1/\tau})+tG(Du_2(z)^{1/\tau})\ \mbox{for almost all}\ z\in\Omega.
	\end{eqnarray*}
	
	Also, recall that $q<\tau$ and so $x\mapsto -x^{q/\tau}$ is convex on $\left[0,+\infty\right)$. Therefore it follows that $j(\cdot)$ is convex. Moreover, Fatou's lemma implies that $j(\cdot)$ is lower semicontinuous.
	
	Now suppose that $\bar{w}_{\lambda}\in W^{1,p}(\Omega)$ is another positive solution of problem \eqref{eqA}. As above we can show that
	$$\bar{w}_{\lambda}\in D_+.$$
	
	Then for all $h\in C^1(\overline{\Omega})$ and for small enogh $|t|\leq 1$, we have
	$$\bar{u}^{\tau}_{\lambda}+th,\bar{w}^{\tau}_{\lambda}+th\in {\rm dom}\,j\ (\mbox{see (\ref{eq58})}).$$
	
	We can easily see that $j(\cdot)$ is G\^ateaux differentiable at $\bar{u}^{\tau}_{\lambda},\bar{w}^{\tau}_{\lambda}$ in the direction $h$. Moreover, via the chain rule and the nonlinear Green's identity (see Gasinski and Papageorgiou \cite[p. 210]{13}), we have
	\begin{eqnarray*}
		&&j'(\bar{u}^{\tau}_{\lambda})(h)=\frac{1}{\tau}\int_{\Omega}\frac{-{\rm div}\, a(D\bar{u}_{\lambda})}{\bar{u}^{\tau-1}_{\lambda}}hdz\\
		&&j'(\bar{w}^{\tau}_{\lambda})(h)=\frac{1}{\tau}\int_{\Omega}\frac{-{\rm div}\, a(D\bar{w}_{\lambda})}{\bar{w}^{\tau-1}_{\lambda}}hdz\ \mbox{for all}\ h\in C^1(\overline{\Omega}).
	\end{eqnarray*}
	
	The convexity of $j(\cdot)$ implies that $j'(\cdot)$ is monotone. Hence
	\begin{eqnarray}\label{eq59}
		0&\leq&\int_{\Omega}\left[\frac{-{\rm div}\, a(D\bar{u}_{\lambda})}{\bar{u}^{\tau-1}_{\lambda}}-\frac{-{\rm div}\, a(D\bar{w}_{\lambda})}{\bar{w}^{\tau-1}_{\lambda}}\right](\bar{u}^{\tau}_{\lambda}-\bar{w}^{\tau}_{\lambda})dz\nonumber\\
		 &=&\int_{\Omega}c_{33}[\bar{w}^{p-\tau}_{\lambda}-\bar{u}^{p-\tau}_{\lambda}](\bar{u}^{\tau}_{\lambda}-\bar{w}^{\tau}_{\lambda})dz.
	\end{eqnarray}
	
	Since $x\mapsto x^{p-\tau}$ is strictly increasing on $\left[0,+\infty\right)$ from (\ref{eq59}) it follows that
	$$\bar{u}_{\lambda}=\bar{w}_{\lambda}.$$
	
	This proves the uniqueness of the positive solution $\bar{u}_{\lambda}\in D_+$ of problem \eqref{eqA}.
	
	The fact that problem \eqref{eqA} is odd, implies that $\bar{v}_{\lambda}=-\bar{u}_{\lambda}\in-D_+$ is the unique negative solution.
\end{proof}

In what follows, for every $\lambda>0$, let $S_+(\lambda)$ (respectively, $S_-(\lambda)$) be the set of positive (respectively, negative) solutions of problem \eqref{eqP}. From Proposition \ref{prop12} and its proof, we know that:
\begin{itemize}
	\item If $\lambda\in(0,\lambda_+)$, then $S_+(\lambda)\neq\emptyset$ and $S_+(\lambda)\subseteq D_+$.
	\item If $\lambda\in(0,\lambda_-)$, then $S_-(\lambda)\neq\emptyset$ and $S_-(\lambda)\subseteq -D_+$.
\end{itemize}

We will use the unique constant sign solutions $\bar{u}_{\lambda}\in D_+$ (respectively, $\bar{v}_{\lambda}\in -D_+$) of the auxiliary problem \eqref{eqA} produced in Proposition \ref{prop13}, to provide a lower bound (respectively, upper bound) for the elements of $S_+(\lambda)$ (respectively, $S_-(\lambda)$).
\begin{prop}\label{prop14}
	If hypotheses $H(a),H(f),H(\beta)$ hold, then
	\begin{itemize}
		\item[(a)] for all $\lambda\in(0,\lambda_+)$ and all $u\in S_+(\lambda)$, we have $\bar{u}_{\lambda}\leq u$;
		\item[(b)] for all $\lambda\in(0,\lambda_-)$ and all $v\in S_-(\lambda)$, we have $v\leq\bar{v}_{\lambda}$.
	\end{itemize}
\end{prop}
\begin{proof}
	\textit{(a)} Let $\lambda\in(0,\lambda_+)$ and $u\in S_+(\lambda)$. We introduce the following Carath\'eodory functions
	\begin{eqnarray}
		&&\hat{k}_+(z,x)=\left\{\begin{array}{ll}
			0&\mbox{if}\ x<0\\
			(-c_{33}+1)x^{p-1}&\mbox{if}\ 0\leq x\leq u(z)\\
			(-c_{33}+1)u(z)^{p-1}&\mbox{if}\ u(z)<x
		\end{array}\right.\label{eq60}\\
		&&\hat{\beta}_+(z,x)=\left\{\begin{array}{ll}
			0&\mbox{if}\ x<0\\
			c_{12}x^{q-1}&\mbox{if}\ 0\leq x\leq u(z)\\
			c_{12}u(z)^{q-1}&\mbox{if}\ u(z)<x
		\end{array}\right.\ \mbox{for all}\ (z,x)\in\partial\Omega\times\RR.\label{eq61}
	\end{eqnarray}
	
	We set $\hat{K}_+(z,x)=\int^x_0\hat{k}_+(z,s)ds$ and $\hat{B}_+(z,x)=\int^x_0\hat{\beta}_+(z,s)ds$ and consider the $C^1$-functional $\hat{\psi}^+_{\lambda}:W^{1,p}(\Omega)\rightarrow\RR$ defined by
	 $$\hat{\psi}^+_{\lambda}(u)=\int_{\Omega}G(Du)dz+\frac{1}{p}||u||^p_p-\int_{\Omega}\hat{K}_+(z,u)dz-\lambda\int_{\partial\Omega}\hat{B}_+(z,u)d\sigma\ \mbox{for all}\ u\in W^{1,p}(\Omega).$$
	
	From Corollary \ref{cor3} and (\ref{eq60}), (\ref{eq61}), we see that $\hat{\psi}^+_{\lambda}$ is coercive. Also, from the Sobolev embedding theorem and the compactness of the trace map, it follows that $\hat{\psi}^+_{\lambda}$ is sequentially weakly lower semicontinuous. So, we can find $\tilde{u}_{\lambda}\in W^{1,p}(\Omega)$ such that
	\begin{equation}\label{eq62}
		\hat{\psi}^+_{\lambda}(\tilde{u}_{\lambda})=\inf[\hat{\psi}^+_{\lambda}(u):u\in W^{1,p}(\Omega)].
	\end{equation}
	
	In fact, since $q<\tau<p$, as in the proof of Proposition \ref{prop12} (see (\ref{eq47}) with $\delta\leq\min\limits_{\overline{\Omega}}u$ and recall that $u\in D_+$), we have
$$\hat{\psi}^+_{\lambda}(\tilde{u}_{\lambda})<0=\hat{\psi}^+_{\lambda}(0)
		\Rightarrow\tilde{u}_{\lambda}\neq 0.
	$$
	
	From (\ref{eq62}) we have
	\begin{eqnarray}\label{eq63}
		&&(\hat{\psi}^+_{\lambda})'(\tilde{u}_{\lambda})=0,\nonumber\\
		&\Rightarrow&\left\langle A(\tilde{u}_{\lambda}),h\right\rangle+\int_{\Omega}|\tilde{u}_{\lambda}|^{p-2}\tilde{u}_{\lambda}hdz=\int_{\Omega}\hat{k}_+(z,\tilde{u}_{\lambda})hdz+\lambda\int_{\partial\Omega}\hat{\beta}_+(z,\tilde{u}_{\lambda})hd\sigma\\
		&&\mbox{for all}\ h\in W^{1,p}(\Omega).\nonumber
	\end{eqnarray}
	
	In (\ref{eq63}) we first choose $h=-\tilde{u}_{\lambda}^-\in W^{1,p}(\Omega)$. Then
	\begin{eqnarray*}
		&&\frac{c_1}{p-1}||D\tilde{u}^-_{\lambda}||^p_p+||\tilde{u}^-_{\lambda}||^p_p\leq 0\ \mbox{(see Corollary \ref{cor3} and (\ref{eq60}), (\ref{eq61}))},\\
		&\Rightarrow&\tilde{u}_{\lambda}\geq 0,\tilde{u}_{\lambda}\neq 0.
	\end{eqnarray*}
	
	Next, in (\ref{eq63}) we choose $h=(\tilde{u}_{\lambda}-u)^+\in W^{1,p}(\Omega)$. Then
	\begin{eqnarray*}
		&&\left\langle A(\tilde{u}_{\lambda}),(\tilde{u}_{\lambda}-u)^+\right\rangle+\int_{\Omega}\tilde{u}^{p-1}_{\lambda}(\tilde{u}_{\lambda}-u)^+dz\\
		 &&=\int_{\Omega}(-c_{33}+1)u^{p-1}(\tilde{u}_{\lambda}-u)^+dz+\lambda\int_{\partial\Omega}c_{12}u^{q-1}(\tilde{u}_{\lambda}-u)^+d\sigma\ (\mbox{see (\ref{eq60}), (\ref{eq61})})\\
		 &&\leq\int_{\Omega}f(z,u)(\tilde{u}_{\lambda}-u)^+dz+\int_{\Omega}u^{p-1}(\tilde{u}_{\lambda}-u)^+dz+\lambda\int_{\partial\Omega}\beta(z,u)(\tilde{u}_{\lambda}-u)^+d\sigma\\
		&&(\mbox{see (\ref{eq55}) and hypothesis}\ H(\beta)(i))\\
		&&=\left\langle A(u),(\tilde{u}_{\lambda}-u)^+\right\rangle+\int_{\Omega}u^{p-1}(\tilde{u}_{\lambda}-u)^+dz\ (\mbox{since}\ u\in S_+(\lambda)),\\
		&\Rightarrow&\left\langle A(\tilde{u}_{\lambda})-A(u),(\tilde{u}_{\lambda}-u)^+\right\rangle+\int_{\Omega}(\tilde{u}_{\lambda}^{p-1}-u^{p-1})(\tilde{u}_{\lambda}-u)^+dz\leq 0,\\
		&\Rightarrow&\tilde{u}_{\lambda}\leq u.
	\end{eqnarray*}
	
	So, we have proved that
	$$\tilde{u}_{\lambda}\in[0,u]=\{y\in W^{1,p}(\Omega):0\leq y(z)\leq u(z)\ \mbox{for almost all}\ z\in\Omega\}.$$
	
	Therefore equation (\ref{eq63}) becomes
	\begin{eqnarray*}
		&&\left\langle A(\tilde{u}_{\lambda}),h\right\rangle+\int_{\Omega}c_{33}\tilde{u}_{\lambda}^{p-1}hdz=\lambda\int_{\partial\Omega}c_{12}\tilde{u}_{\lambda}^{q-1}hd\sigma\ \mbox{for all}\ h\in W^{1,p}(\Omega),\\
		&\Rightarrow&-{\rm div}\, a(D\tilde{u}_{\lambda}(z))+c_{33}\tilde{u}_{\lambda}(z)^{p-1}=0\ \mbox{for almost all}\ z\in\Omega,\\
		&&\frac{\partial \tilde{u}_{\lambda}}{\partial n_a}=\lambda c_{12}\tilde{u}_{\lambda}^{q-1}\ \mbox{on}\ \partial\Omega,\tilde{u}_{\lambda}\geq 0,\tilde{u}_{\lambda}\neq 0\ \mbox{(see Papageorgiou and R\u adulescu \cite{28})},\\
		&\Rightarrow&\tilde{u}_{\lambda}=\bar{u}_{\lambda}\ (\mbox{see Proposition \ref{prop13}}).
	\end{eqnarray*}
	
	Since $u\in S_+(\lambda)$ is arbitrary, we conclude that
	$$\bar{u}_{\lambda}\leq u\ \mbox{for all}\ u\in S_+(\lambda).$$
	
	\textit{(b)} In a similar fashion, we show that if $\lambda\in(0,\lambda_-)$, then $v\leq\bar{v}_{\lambda}$ for all $v\in S_-(\lambda)$.
\end{proof}

Using this proposition, we can produce the desired extremal constant sign solutions for problem \eqref{eqP}.

As in Filippakis and Papageorgiou \cite{10} (see Lemmata 4.1 and 4.2), we have:
\begin{itemize}
	\item	$S_+(\lambda)$ is downward directed, that is, if $u_1,u_2\in S_+(\lambda)$, then we can find $u\in S_+(\lambda)$ such that $u\leq u_1,u\leq u_2$.
	\item $S_-(\lambda)$ is upward directed, that is, if $v_1,v_2\in S_-(\lambda)$, then we can find $v\in S_-(\lambda)$ such that $v_1\leq v,v_2\leq v$.
\end{itemize}
\begin{prop}\label{prop15}
	If hypotheses $H(a),H(f),H(\beta)$ hold, then
	\begin{itemize}
		\item[(a)] for every $\lambda\in(0,\lambda_+)$ problem \eqref{eqP} has a smallest positive solution
		$$u^*_{\lambda}\in D_+;$$
		\item[(b)] for every $\lambda\in(0,\lambda_-)$ problem \eqref{eqP} has a biggest negative solution
		$$v^*_{\lambda}\in-D_+.$$
	\end{itemize}
\end{prop}
\begin{proof}
	\textit{(a)} Using Lemma 3.9, p. 178 of Hu and Papageorgiou \cite{18}, we can find a decreasing sequence $\{u_n\}_{n\geq 1}\subseteq S_+(\lambda)$ such that
	$$\inf S_+(\lambda)=\inf\limits_{n\geq 1}u_n.$$
	
	We have for all $h\in W^{1,p}(\Omega)$ and all $n\in \NN$
	\begin{eqnarray}\label{eq64}
		\left\langle A(u_n),h\right\rangle=\int_{\Omega}f(z,u_n)hdz+\lambda\int_{\partial\Omega}\beta(z,u_n)hd\sigma\,.
	\end{eqnarray}
	
	Since $0\leq u_n\leq u_1$ for all $n\in\NN$, using (\ref{eq64}), Corollary \ref{cor3}, hypothesis $H(f)(i)$ and (\ref{eq19}), we can infer that $\{u_n\}_{n\geq 1}\subseteq W^{1,p}(\Omega)$ is bounded. So, we may assume that
	\begin{equation}\label{eq65}
		u_n\stackrel{w}{\rightarrow}u^*_{\lambda}\ \mbox{in}\ W^{1,p}(\Omega)\ \mbox{and}\ u_n\rightarrow u^*_{\lambda}\ \mbox{in}\ L^r(\Omega)\ \mbox{and in}\ L^q(\partial\Omega).
	\end{equation}
	
	In (\ref{eq64}) we choose $h=u_n-u^*_{\lambda}\in W^{1,p}(\Omega)$, pass to the limit as $n\rightarrow\infty$ and use (\ref{eq65}). Then we obtain
	\begin{eqnarray}\label{eq66}
		&&\lim\limits_{n\rightarrow\infty}\left\langle A(u_n),u_n-u^*_{\lambda}\right\rangle=0,\nonumber\\
		&\Rightarrow&u_n\rightarrow u^*_{\lambda}\ \mbox{in}\ W^{1,p}(\Omega)\ (\mbox{see Proposition \ref{prop4}}).
	\end{eqnarray}
	
	So, passing to the limit as $n\rightarrow\infty$ in (\ref{eq64}) and using (\ref{eq66}), we have
	\begin{eqnarray*}
		&&\left\langle A(u^*_{\lambda}),h\right\rangle=\int_{\Omega}f(z,u^*_{\lambda})hdz+\lambda\int_{\partial\Omega}\beta(z,u^*_{\lambda})hd\sigma\ \mbox{for all}\ h\in W^{1,p}(\Omega),\\
		&\Rightarrow&u^*_{\lambda}\ \mbox{is a nonnegative solution of \eqref{eqP}}.
	\end{eqnarray*}
	
	From Proposition \ref{prop14} we know that
	\begin{eqnarray*}
		&&\bar{u}_{\lambda}\leq u_n\ \mbox{for all}\ n\in\NN,\\
		&\Rightarrow&\bar{u}_{\lambda}\leq u^*_{\lambda}\ (\mbox{see (\ref{eq66})}),\\
		&\Rightarrow&u^*_{\lambda}\in S_+(\lambda)\ \mbox{and}\ u^*_{\lambda}=\inf S_+(\lambda).
	\end{eqnarray*}
	
	\textit{(b)} Reasoning in a similar fashion, we show that for all $\lambda\in(0,\lambda_-)$ problem \eqref{eqP} has a biggest negative solution $v^*_{\lambda}\in S_-(\lambda)$.
\end{proof}

In Section 4, using these extremal constant sign solutions, we will produce a nodal (sign changing) solution for problem \eqref{eqP}. For the moment, in the remaining part of this section we examine the maps
\begin{eqnarray}
	&&\lambda\mapsto u^*_{\lambda}\ \mbox{from}\ (0,\lambda_+)\ \mbox{into}\ C^1(\overline{\Omega})\label{eq67},\\
	&&\lambda\mapsto v^*_{\lambda}\ \mbox{from}\ (0,\lambda_-)\ \mbox{into}\ C^1(\overline{\Omega}).\label{eq68}
\end{eqnarray}

The next proposition will be used to prove the monotonicity properties of the maps in (\ref{eq67}), (\ref{eq68}).
\begin{prop}\label{prop16}
	If hypotheses $H(a),H(f)',H(\beta)$ hold, then
	\begin{itemize}
		\item[(a)] given $\lambda,\mu\in(0,\lambda_+)$ with $\mu<\lambda$ and $u_{\lambda}\in S_+(\lambda)$, we can find $u_{\mu}\in S_+(\mu)$ such that
		$$u_{\lambda}-u_{\mu}\in {\rm int}\, \hat{C}_+;$$
		\item[(b)] given $\lambda,\mu\in(0,\lambda_-)$ with $\mu<\lambda$ and $v_{\lambda}\in S_-(\lambda)$, we can find $u_{\mu}\in S_-(\mu)$ such that
		$$v_{\mu}-v_{\lambda}\in {\rm int}\,\hat{C}_+.$$
	\end{itemize}
\end{prop}
\begin{proof}
	\textit{(a)} We introduce the following Carath\'eodory functions
	\begin{eqnarray}
		&&e_+(z,x)=\left\{\begin{array}{ll}
			f(z,x^+)+(x^+)^{p-1}&\mbox{if}\ x\leq u_{\lambda}(z)\\
			f(z,u_{\lambda}(z))+u_{\lambda}(z)^{p-1}&\mbox{if}\ u_{\lambda}(z)<x,
		\end{array}\right.\label{eq69}\\
		&&d^+_{\mu}(z,x)=\left\{\begin{array}{ll}
			\mu\beta(z,x)&\mbox{if}\ x\leq u_{\lambda}(z)\\
			\mu\beta(z,u_{\lambda}(z))&\mbox{if}\ u_{\lambda}(z)<x
		\end{array}\right.\ \mbox{for all}\ (z,x)\in\partial\Omega\times\RR.\label{eq70}
	\end{eqnarray}
	
	We set
	$$E_+(z,x)=\int^x_0e_+(z,s)ds\ \mbox{and}\ D^+_{\mu}(z,x)=\int^x_0d^+_{\mu}(z,s)ds$$
	and consider the $C^1$-functional $\vartheta^+_{\mu}(z,x):W^{1,p}(\Omega)\rightarrow\RR$ defined by
	\begin{eqnarray*}
		 &&\vartheta^+_{\mu}(u)=\int_{\Omega}G(Du)dz+\frac{1}{p}||u||^p_p-\int_{\Omega}E_+(z,u)dz-\int_{\partial\Omega}D^+_{\mu}(z,u)d\sigma\\
		&&\mbox{for all}\ u\in W^{1,p}(\Omega).
	\end{eqnarray*}
	
	From Corollary \ref{cor3} and (\ref{eq69}), (\ref{eq70}), it is clear that the function $\vartheta^+_{\mu}$ is coercive. Also, it is sequentially weakly lower semicontinuous. So, we can find $u_{\mu}\in W^{1,p}(\Omega)$ such that
	\begin{equation}\label{eq71}
		\vartheta^+_{\mu}(u_{\mu})=\inf[\vartheta^+_{\mu}(u):u\in W^{1,p}(\Omega)].
	\end{equation}
	
	Since $q<\tau<p$, we have
	\begin{eqnarray*}
		&&\vartheta^+_{\mu}(u_{\mu})<0=\vartheta^+_{\mu}(0)\ (\mbox{see the proof of Proposition \ref{prop12}}),\\
		&\Rightarrow&u_{\mu}\neq 0.
	\end{eqnarray*}
	
	From (\ref{eq71}) we have
	\begin{eqnarray}\label{eq72}
		&&(\vartheta^+_{\mu})'(u_{\mu})=0,\nonumber\\
		&\Rightarrow&\left\langle A(u_{\mu}),h\right\rangle+\int_{\Omega}|u_{\mu}|^{p-2}u_{\mu}hdz=\int_{\Omega}e_+(z,u_{\mu})hdz+\int_{\partial\Omega}d^+_{\mu}(z,u_{\mu})hd\sigma\\
		&&\mbox{for all}\ h\in W^{1,p}(\Omega).\nonumber
	\end{eqnarray}
	
	In (\ref{eq72}) we first choose $h=-u^-_{\mu}\in W^{1,p}(\Omega)$. From Lemma \ref{lem2} and (\ref{eq69}), (\ref{eq70}) we have
	\begin{eqnarray*}
		&&\frac{c_1}{p-1}||Du^-_{\mu}||^p_p+||u^-_{\mu}||^p_p\leq 0,\\
		&\Rightarrow&u_{\mu}\geq 0,u_{\mu}\neq 0.
	\end{eqnarray*}
	
	Next, in (\ref{eq72}) we choose $h=(u_{\mu}-u_{\lambda})^+\in W^{1,p}(\Omega)$. Then
	\begin{eqnarray*}
		&&\left\langle A(u_{\mu}),(u_{\mu}-u_{\lambda})^+\right\rangle+\int_{\Omega}u^{p-1}_{\mu}(u_{\mu}-u_{\lambda})^+dz\\
		 &=&\int_{\Omega}f(z,u_{\lambda})(u_{\mu}-u_{\lambda})^+dz+\int_{\Omega}u_{\lambda}^{p-1}(u_{\mu}-u_{\lambda})^+dz+\mu\int_{\partial\Omega}\beta(z,u_{\lambda})(u_{\mu}-u_{\lambda})^+d\sigma\\
		&&(\mbox{see (\ref{eq69}),(\ref{eq70})})\\
		 &\leq&\int_{\Omega}f(z,u_{\lambda})(u_{\mu}-u_{\lambda})^+dz+\int_{\Omega}u_{\lambda}^{p-1}(u_{\mu}-u_{\lambda})^+dz+\lambda\int_{\partial\Omega}\beta(z,u_{\lambda})(u_{\lambda}-u_{\mu})^+d\sigma\\
		&&(\mbox{since}\ \mu<\lambda,\ \mbox{see hypothesis}\ H(\beta)(i))\\
		&=&\left\langle A(u_{\lambda}),(u_{\mu}-u_{\lambda})^+\right\rangle+\int_{\Omega}u_{\lambda}^{p-1}(u_{\mu}-u_{\lambda})^+dz\ (\mbox{since}\ u_{\lambda}\in S_+(\lambda)),\\
		\Rightarrow&&\left\langle A(u_{\mu})-A(u_{\lambda}),(u_{\mu}-u_{\lambda})^+\right\rangle+\int_{\Omega}(u^{p-1}_{\mu}-u_{\lambda}^{p-1})(u_{\mu}-u_{\lambda})^+dz\leq 0,\\
		\Rightarrow&&u_{\mu}\leq u_{\lambda}.
	\end{eqnarray*}
	
	So, we have proved that
	\begin{equation}\label{eq73}
		u_{\mu}\in[0,u_{\lambda}].
	\end{equation}
	
	Invoking (\ref{eq69}), (\ref{eq70}), (\ref{eq73}), equation (\ref{eq72}) becomes
	\begin{eqnarray}\label{eq74}
		&&\left\langle A(u_{\mu}),h\right\rangle=\int_{\Omega}f(z,u_{\mu})hdz+\mu\int_{\partial\Omega}\beta(z,u_{\mu})hd\sigma\ \mbox{for all}\ h\in W^{1,p}(\Omega),\nonumber\\
		&\Rightarrow&-{\rm div}\, a(Du_{\mu}(z))=f(z,u_{\mu}(z))\ \mbox{for almost all}\ z\in\Omega,\ \frac{\partial u_{\mu}}{\partial n_a}=\mu\beta(z,u_{\mu})\ \mbox{on}\ \partial\Omega\\
		&&(\mbox{see Papageorgiou and R\u adulescu \cite{28}}),\nonumber\\
		&\Rightarrow&u_{\mu}\in S_+(\mu).\nonumber
	\end{eqnarray}
	
	Evidently, $u_{\mu}\neq u_{\lambda}$ (recall that $\mu<\lambda$ and use hypothesis $H(\beta)(i)$). Then hypothesis $H(f)'(iv)$ implies that
	\begin{eqnarray*}
		&&-{\rm div}\, a (Du_{\mu}(z))=f(z,u_{\mu}(z))=g_{\mu}(z)\\
		&&\leq g_{\lambda}(z)=f(z,u_{\lambda}(z))=-{\rm div}\, a(Du_{\lambda}(z))\ \mbox{for almost all}\ z\in\Omega,
	\end{eqnarray*}
	with $g_{\mu},g_{\lambda}\in L^{\infty}(\Omega)$ and $g_{\mu}\not\equiv g_{\lambda}$. Also, we have
	$$\left.\frac{\partial u_{\mu}}{\partial n}\right|_{\partial\Omega}>0,\ \ \left.\frac{\partial u_{\lambda}}{\partial n}\right|_{\partial\Omega}>0\ \mbox{see hypothesis}\ H(\beta)(i)).$$
	
	Therefore, we can use Proposition \ref{prop6} and infer that
	$$u_{\lambda}-u_{\mu}\in {\rm int}\,\hat{C}_+\ (\mbox{that is},\ u_{\mu}\in {\rm int}_{C^1(\overline{\Omega})}[0,u_{\lambda}]).$$
	
	\textit{(b)} For this part, we consider the following Carath\'eodory functions
	\begin{eqnarray}
		&&e_-(z,x)=\left\{\begin{array}{ll}
			f(z,v_{\lambda}(z))+|v_{\lambda}(z)|^{p-2}v_{\lambda}(z)&\mbox{if}\ x<v_{\lambda}(z)\\
			f(z,-x^-)+|x|^{p-2}(-x^-)&\mbox{if}\ v_{\lambda}(z)\leq x,
		\end{array}\right.\label{eq75}\\
		&&d^-_{\mu}(z,x)=\left\{\begin{array}{ll}
			\mu\beta(z,v_{\lambda}(z))&\mbox{if}\ x<v_{\lambda}(z)\\
			\mu\beta(z,-x^-)&\mbox{if}\ v_{\lambda}(z)\leq x
		\end{array}\right.\ \mbox{for all}\ (z,x)\in\partial\Omega\times\RR.\label{eq76}
	\end{eqnarray}
	
	We set $E_-(z,x)=\int^x_0e_-(z,s)ds$ and $D^-_{\mu}(z,x)=\int^x_0d^-_{\mu}(z,s)ds$ and consider the $C^1$-functional $\vartheta^-_{\mu}:W^{1,p}(\Omega)\rightarrow\RR$ defined by
	 $$\vartheta^-_{\mu}(u)=\int_{\Omega}G(Du)dz+\frac{1}{p}||u||^p_p-\int_{\Omega}E_-(z,u)dz-\int_{\partial\Omega}D^-_{\mu}(z,u)d\sigma\ \mbox{for all}\ u\in W^{1,p}(\Omega).$$
	
	Reasoning as in part \textit{(a)}, we produce some $v_{\mu}\in S_-(\mu)$ such that
	$$v_{\mu}-v_{\lambda}\in {\rm int}\, \hat{C}_+\ (\mbox{that is,}\ v_{\mu}\in {\rm int}_{C^1(\overline{\Omega})}[v_{\lambda},0]).$$
\end{proof}

Now we can establish the monotonicity and continuity properties of the two maps defined in (\ref{eq67}) and (\ref{eq68}).
\begin{prop}\label{prop17}
	If hypotheses $H(a),H(f)',H(\beta)$ hold, then
	\begin{itemize}
		\item[(a)] the map $\lambda\mapsto u^*_{\lambda}$ from $(0,\lambda_+)$ into $C^1(\overline{\Omega})$ is strictly increasing in the sense that $\mu<\lambda\Rightarrow u^*_{\lambda}-u^*_{\mu}\in {\rm int}\,\hat{C}_+$ and is left continuous;
		\item[(b)] the map $\lambda\mapsto v^*_{\lambda}$ from $(0,\lambda_-)$ into $C^1(\overline{\Omega})$ is strictly decreasing in the sense that $\mu<\lambda\Rightarrow v^*_{\mu}-v^*_{\lambda}\in {\rm int}\,\hat{C}_+$ and is right continuous.
	\end{itemize}
\end{prop}
\begin{proof}
	\textit{(a)} Let $\mu,\lambda\in(0,\lambda_+)$ with $\mu<\lambda$. From Proposition \ref{prop15}, we know that problem \eqref{eqP} has a smallest positive solution $u^*_{\lambda}\in S_+(\lambda)$. Invoking Proposition \ref{prop16}, we can find $u_{\mu}\in S_+(\mu)$ such that
	\begin{eqnarray*}
		&&u^*_{\lambda}-u_{\mu}\in {\rm int}\,\hat{C}_+,\\
		&\Rightarrow&u^*_{\lambda}-u^*_{\mu}\in {\rm int}\,\hat{C}_+\ (\mbox{see Proposition \ref{prop15}})\\
		&\Rightarrow&\lambda\mapsto u^*_{\lambda}\ \mbox{is strictly increasing as claimed by the proposition}.
	\end{eqnarray*}
	
	Next, let $\{\lambda_n,\lambda\}_{n\geq 1}\subseteq(0,\lambda_+)$ and assume that $\lambda_n\rightarrow\lambda^-$. We have
	$$0<\tilde{\lambda}\leq\lambda_n\leq\hat{\lambda}<\lambda_+\ \mbox{for all}\ n\in\NN.$$
	
	Then from Proposition \ref{prop15} and the first part of the proof, we have
	\begin{equation}\label{eq77}
		0\leq u^*_{\tilde{\lambda}}\leq u^*_{\lambda_n}\leq u^*_{\hat{\lambda}}\ \mbox{for all}\ n\in\NN.
	\end{equation}
	
	Hence the nonlinear regularity theory of Lieberman \cite{21} implies that there exist $\alpha\in(0,1)$ and $c_{37}>0$ such that
	\begin{equation}\label{eq78}
		u^*_{\lambda_n}\in C^{1,\alpha}(\overline{\Omega})\ \mbox{and}\ ||u^*_{\lambda_n}||_{C^{1,\alpha}(\overline{\Omega})}\leq c_{37}\ \mbox{for all}\ n\in\NN.
	\end{equation}
	
	Exploiting the compact embedding of $C^{1,\alpha}(\overline{\Omega})$ into $C^1(\overline{\Omega})$ and by passing to a subsequence if necessary, we can say that
	\begin{equation}\label{eq79}
		u^*_{\lambda_n}\rightarrow\tilde{u}_{\lambda}\ \mbox{in}\ C^1(\overline{\Omega}).
	\end{equation}
	
	Evidently, we have
	$$u^*_{\tilde{\lambda}}\leq\tilde{u}_{\lambda}\ \mbox{and}\ \tilde{u}_{\lambda}\in S_+(\lambda)\ (\mbox{see (\ref{eq77}), (\ref{eq79})}).$$
	
	Suppose that $\tilde{u}_{\lambda}\neq u^*_{\lambda}$. Then we can find $z_0\in\Omega$ such that
	\begin{eqnarray*}
		&&u^*_{\lambda}(z_0)<\tilde{u}_{\lambda}(z_0),\\
		&\Rightarrow&u^*_{\lambda}(z_0)<u^*_{\lambda_n}(z_0)\ \mbox{for all}\ n\geq n_0\ (\mbox{see (\ref{eq79})}).
	\end{eqnarray*}
	
	This contradicts the first part (that is, the ``monotonicity'' part) of the proof. So, $\tilde{u}_{\lambda}=u^*_{\lambda}$ and now by Urysohn's criterion we conclude that for the initial sequence we have
	\begin{eqnarray*}
		&&u^*_{\lambda_n}\rightarrow u^*_{\lambda}\ \mbox{in}\ C^1(\overline{\Omega}),\\
		&\Rightarrow&\lambda\mapsto u^*_{\lambda}\ \mbox{is left continuous from}\ (0,\lambda_+)\ \mbox{into}\ C^1(\overline{\Omega}).
	\end{eqnarray*}
	
	\textit{(b)} In a similar fashion we show that the map $\lambda\mapsto v^*_{\lambda}$ from $(0,\lambda_-)$ into $C^1(\overline{\Omega})$ is strictly decreasing (in the sense described in the proposition) and right continuous.
\end{proof}

\section{Nodal Solutions}
	
	In this section we turn our attention to the existence of nodal solutions. To do this, we will use a combination of variational methods and Morse theory. So, we start with the computation of the critical groups at the origin of the energy (Euler) functional $\varphi_{\lambda}$.
\begin{prop}\label{prop18}
	If hypotheses $H(a),H(f)',H(\beta)$ hold, $\lambda>0$, and $K_{\varphi_{\lambda}}$ is finite, then $C_k(\varphi_{\lambda},0)=0$ for all $k\in\NN_0$.
\end{prop}
\begin{proof}
	Hypothesis $H(a)(iv)$ and Corollary \ref{cor3} imply that
	\begin{equation}\label{eq80}
		G(y)\leq c_{38}(|y|^{\tau}+|y|^p)\ \mbox{for all}\ y\in\RR^N\ \mbox{and some}\ c_{38}>0.
	\end{equation}
	
	Also, hypotheses $H(f)'(i),(ii),(iii)$ (see also (\ref{eq18})) imply that
	\begin{equation}\label{eq81}
		F(z,x)\geq c_{39}|x|^{\eta}-c_{40}|x|^p\ \mbox{for almost all}\ z\in\Omega\ \mbox{and all}\ x\in\RR,\ \mbox{with}\ c_{39},\ c_{40}>0.
	\end{equation}
	
	Moreover, from (\ref{eq19}) we have
	\begin{equation}\label{eq82}
		|B(z,x)|\leq c_{41}|x|^q\ \mbox{for all}\ (z,x)\in\partial\Omega\times\RR\ \mbox{and some}\ c_{41}>0.
	\end{equation}
	
	For $u\in W^{1,p}(\Omega)$ and $t>0$, we have
	\begin{eqnarray}\label{eq83}
		 \varphi_{\lambda}(tu)&=&\int_{\Omega}G(tDu)dz-\int_{\Omega}F(z,tu)dz-\lambda\int_{\partial\Omega}B(z,tu)d\sigma\nonumber\\
		 &\leq&c_{38}(t^{\tau}||Du||^{\tau}_{\tau}+t^p||Du||^p_p)-c_{39}t^{\eta}||u||^{\eta}_{\eta}+c_{40}t^p||u||^p_p-\lambda c_{41}t^q||u||^q_{L^q(\partial\Omega)}\\
		&&(\mbox{see (\ref{eq80}), (\ref{eq81}), (\ref{eq82})}).\nonumber
	\end{eqnarray}
	
	Since $q<\tau<p<\eta$, from (\ref{eq83}) we see that we can find $t^*=t^*(u)\in(0,1)$ such that
	\begin{equation}\label{eq84}
		\varphi_{\lambda}(tu)<0\ \mbox{for all}\ t\in(0,t^*).
	\end{equation}
	
	Now, let $u\in W^{1,p}(\Omega)$ with $0<||u||\leq 1$ and $\varphi_{\lambda}(u)=0$. Then
	\begin{eqnarray}\label{eq85}
		\left.\frac{d}{dt}\varphi_{\lambda}(tu)\right|_{t=1}&=&\left\langle \varphi'_{\lambda}(u),u\right\rangle\ (\mbox{by the chain rule})\nonumber\\
		 &=&\int_{\Omega}(a(Du),Du)_{\RR^N}dz-\int_{\Omega}f(z,u)udz-\lambda\int_{\partial\Omega}\beta(z,u)ud\sigma\nonumber\\
		&=&\int_{\Omega}[(a(Du),Du)_{\RR^N}-\tau G(Du)]dz\nonumber\\
		&&+\int_{\Omega}[\tau F(z,u)-f(z,u)u]dz+\lambda\int_{\partial\Omega}[\tau B(z,u)-\beta(z,u)u]d\sigma\\
		&&(\mbox{since}\ \varphi_{\lambda}(u)=0).\nonumber
	\end{eqnarray}
	
	Hypothesis $H(a)(iv)$ implies that
	\begin{equation}\label{eq86}
		\int_{\Omega}[(a(Du),Du)_{\RR^N}-\tau G(Du)]dz\geq c_{4}||Du||^p_p.
	\end{equation}
	
	Also, hypotheses $H(f)'(i),(iii)$ imply that given $\epsilon>0$, we can find $c_{42}=c_{42}(\epsilon)>0$ such that
	\begin{eqnarray}\label{eq87}
		&&\tau F(z,x)-f(z,x)x\geq -\epsilon|x|^p-c_{42}|x|^r\ \mbox{for almost all}\ z\in\Omega\ \mbox{and all}\ x\in\RR,\nonumber\\
		&\Rightarrow&\int_{\Omega}[\tau F(z,u)-f(z,u)u]dz\geq-\epsilon||u||^p_p-c_{42}||u||^r_r\,.
	\end{eqnarray}
	
	Finally, from hypothesis $H(\beta)(iv)$, we have
	$$\lambda\int_{\Omega}[\tau B(z,u)-\beta(z,u)u]d\sigma\geq\lambda c_{14}||u||^q_{L^q(\partial\Omega)}.$$
	
	Since $q<p$, for all $||u||_{L^q(\partial\Omega)}\leq 1$ we have
	\begin{eqnarray}\label{eq88}
		&&||u||^q_{L^q(\partial\Omega)}\geq||u||^p_{L^q(\partial\Omega)},\nonumber\\
		&\Rightarrow&\lambda\int_{\partial\Omega}[\tau B(z,u)-\beta(z,u)u]d\sigma\geq \lambda c_{14}||u||^p_{L^q(\partial\Omega).}
	\end{eqnarray}
	
	From Proposition \ref{eq7} (see also the remark following that proposition), we know that
	$$v\mapsto ||Dv||_p+||v||_{L^q(\partial\Omega)},\ v\in W^{1,p}(\Omega),$$
	is an equivalent norm on the Sobolev space $W^{1,p}(\Omega)$.
	
	So, returning to (\ref{eq85}) and using (\ref{eq86}), (\ref{eq87}) and (\ref{eq88}) and choosing small $\epsilon>0$, we see that for all $u\in W^{1,p}(\Omega)$ with $0<||u||\leq 1$ and $||u||_{L^q(\partial\Omega)}\leq 1$, $\varphi_{\lambda}(u)=0$, we have
	\begin{equation}\label{eq89}
		\left.\frac{d}{dt}\varphi_{\lambda}(tu)\right|_{t=1}\geq c_{43}||u||^p-c_{44}||u||^r\ \mbox{for some}\ c_{43},\ c_{44}>0.
	\end{equation}
	
	Recall that $p<r$. Choosing $\rho\in(0,1)$ small, we have
	\begin{equation}\label{eq90}
		\left.\frac{d}{dt}\varphi_{\lambda}(tu)\right|_{t=1}>0\ \mbox{for all}\ 0<||u||\leq\rho,\varphi_{\lambda}(u)=0
	\end{equation}
	(recall that via the trace map, $W^{1,p}(\Omega)$ is embedded continuously into $L^q(\partial\Omega)$).
	
	Now consider $u\in W^{1,p}(\Omega)$ with $0<||u||\leq\rho,\ \varphi_{\lambda}(u)=0$. We will show that
	\begin{equation}\label{eq91}
		\varphi_{\lambda}(tu)\leq 0\ \mbox{for all}\ t\in[0,1].
	\end{equation}
	
	If (\ref{eq91}) is not true, then we can find $t_0\in(0,1)$ such that
	$$\varphi_{\lambda}(t_0u)>0.$$
	
	Since $\varphi_{\lambda}(u)=0$ and $\varphi_{\lambda}(\cdot)$ is continuous, we have
	$$t_*=\min\{t\in[t_0,1]:\varphi_{\lambda}(tu)=0\}>t_0>0.$$
	
	We have
	\begin{equation}\label{eq92}
		\varphi_{\lambda}(tu)>0\ \mbox{for all}\ t\in\left[t_0,t_*\right).
	\end{equation}
	
	We set $y=t_*u$. Then $0<||y||\leq||u||\leq\rho$ and $\varphi_{\lambda}(y)=0$. So, it follows from (\ref{eq90})   that
	\begin{equation}\label{eq93}
		\left.\frac{d}{dt}\varphi_{\lambda}(ty)\right|_{t=1}>0.
	\end{equation}
	
	From (\ref{eq92}) we have
	$$\varphi_{\lambda}(y)=\varphi_{\lambda}(t_*u)=0<\varphi_{\lambda}(tu)\ \mbox{for all}\ t\in\left[t_0,t_*\right)$$
	and this implies that
	\begin{equation}\label{eq94}
		 \left.\frac{d}{dt}\varphi_{\lambda}(ty)\right|_{t=1}=t_*\left.\frac{d}{dt}\varphi_{\lambda}(tu)\right|_{t=t_*}=t_*\lim\limits_{t\rightarrow t^-_*}\frac{\varphi_{\lambda}(tu)}{t-t_*}\leq 0.
	\end{equation}
	
	Comparing (\ref{eq93}) and (\ref{eq94}), we obtain a contradiction. This proves (\ref{eq91}).
	
	We can always choose $\rho\in(0,1)$ small enough so that $K_{\varphi_{\lambda}}\cap\bar{B}_{\rho}=\{0\}$ (here, $\bar{B}_{\rho}=\{v\in W^{1,p}(\Omega):||v||\leq\rho\}$). We consider the deformation $h:[0,1]\times(\varphi_{\lambda}^{0}\cap\bar{B}_{\rho})\rightarrow \varphi_{\lambda}^{0}\cap\bar{B}_{\rho}$ defined by
	$$h(t,u)=(1-t)u\ \mbox{for all}\ (t,u)\in[0,1]\times(\varphi_{\lambda}^{0}\cap\bar{B}_{\rho}).$$
	
	Using (\ref{eq91}), we see easily that this is a well-defined deformation and it implies that the set $\varphi_{\lambda}^{0}\cap\bar{B}_{\rho}$ is contractible in itself.
	
	Let $u\in\bar{B}_{\rho}$ and assume that $\varphi_{\lambda}(u)>0$. We will show that there is a unique $t(u)\in(0,1)$ such that
	\begin{equation}\label{eq95}
		\varphi_{\lambda}(t(u)u)=0.
	\end{equation}
	
	From (\ref{eq84}) and Bolzano's theorem, we see that there exists $t(u)\in(0,1)$ such that (\ref{eq95}) holds. We need to show that $t(u)\in(0,1)$ is unique. Arguing by contradiction, suppose we can find
	\begin{equation}\label{eq96}
		0<t_1=t(u)_1<t_2=t(u)_2<1\ \mbox{such that}\ \varphi_{\lambda}(t_1u)=\varphi_{\lambda}(t_2u)=0.
	\end{equation}
	
	From (\ref{eq91}) we have
	\begin{eqnarray*}
		&&\varphi_{\lambda}(tt_2u)\leq 0\ \mbox{for all}\ t\in[0,1],\\
		&\Rightarrow&\frac{t_1}{t_2}\in(0,1)\ \mbox{is a maximizer of the function}\ t\mapsto \varphi_{\lambda}(tt_2u),\\
		&\Rightarrow&\frac{t_1}{t_2}\ \left.\frac{d}{dt}\varphi_{\lambda}(tt_2u)\right|_{t=\frac{t_1}{t_2}}=\left.\frac{d}{dt}\varphi_{\lambda}(tt_1u)\right|_{t=1}=0,
	\end{eqnarray*}
	which contradicts (\ref{eq90}). Therefore $t(u)\in(0,1)$ for which (\ref{eq95}) holds is indeed unique. Then
	$$\varphi_{\lambda}(tu)<0\ \mbox{for all}\ t\in(0,t(u))\ \mbox{(see (\ref{eq84})) and}\ \varphi_{\lambda}(tu)>0\ \mbox{for all}\ t\in\left(t(u),1\right].$$
	
	Consider the function $\vartheta:\bar{B}_{\rho}\backslash\{0\}\rightarrow[0,1]$ defined by
	$$\vartheta(u)=\left\{\begin{array}{ll}
		1&\mbox{if}\ u\in\bar{B}_{\rho}\backslash\{0\},\ \varphi_{\lambda}(u)\leq 0\\
		t(u)&\mbox{if}\ u\in\bar{B}_{\rho}\backslash\{0\},\ \varphi_{\lambda}(u)>0.
	\end{array}\right.$$
	
	It is easy to see that $\vartheta(\cdot)$ is continuous. Now let $d:\bar{B}_{\rho}\backslash\{0\}\rightarrow(\varphi_{\lambda}^{\circ}\cap\bar{B}_{\rho})\backslash\{0\}$ be the map defined by
	$$d(u)=\left\{\begin{array}{ll}
		u&\mbox{if}\ u\in\bar{B}_{\rho}\backslash\{0\},\
 \varphi_{\lambda}(u)\leq 0\\
		\vartheta(u)u&\mbox{if}\ u\in\bar{B}_{\rho}\backslash\{0\},\ \varphi_{\lambda}(u)>0.
	\end{array}\right.$$
	
	The continuity of $\vartheta(\cdot)$ implies the continuity of $d(\cdot)$. Note that
	$$d|_{(\varphi_{\lambda}^{0}\cap\bar{B}_{\rho})\backslash\{0\}}={\rm id}|_{(\varphi_{\lambda}^{0}\cap\bar{B}_{\rho})\backslash\{0\}}.$$
	
	Hence $(\varphi_{\lambda}^{0}\cap\bar{B}_{\rho})\backslash\{0\}$ is a retract of $\bar{B}_{\rho}\backslash\{0\}$ and the latter is contractible. Thus so is the set $(\varphi_{\lambda}^{0}\cap\bar{B}_{\rho})\backslash\{0\}$. Recall that we have established earlier that $\varphi_{\lambda}^{0}\cap\bar{B}_{\rho}$ is contractible. Therefore we have
	\begin{eqnarray*}
		&&H_k(\varphi_{\lambda}^{0}\cap\bar{B}_{\rho},(\varphi_{\lambda}^{0}\cap\bar{B}_{\rho})\backslash\{0\})=0\ \mbox{for all}\ k\in\NN_0\\
		&&(\mbox{see Motreanu, Motreanu and Papageorgiou \cite[p. 147]{23}})\\
		&\Rightarrow&C_k(\varphi_{\lambda},0)=0\ \mbox{for all}\ k\in\NN_0.
	\end{eqnarray*}
\end{proof}

Recall that $\lambda_0=\min\{\lambda_+,\lambda_-\}$. Next, we show that for every $\lambda\in(0,\lambda_0)$ problem \eqref{eqP} admits a nodal solution.
\begin{prop}\label{prop19}
	If hypotheses $H(a),H(f)',H(\beta)$ hold and $\lambda\in(0,\lambda_0)$, then problem \eqref{eqP} admits a nodal solution $y_0\in[v^*_{\lambda},u^*_{\lambda}]\cap C^1(\overline{\Omega})$.
\end{prop}
\begin{proof}
	Let $u^*_{\lambda}\in D_+$ and $v^*_{\lambda}\in-D_+$ be the two extremal constant sign solutions of problem \eqref{eqP} produced in Proposition \ref{prop15}. We introduce the following Carath\'eodory functions
	\begin{eqnarray}
		&&k(z,x)=\left\{\begin{array}{ll}
			f(z,v^*_{\lambda}(z))+|v^*_{\lambda}(z)|^{p-2}v^*_{\lambda}(z)&\mbox{if}\ x<v^*_{\lambda}(z)\\
			f(z,x)+|x|^{p-2}x&\mbox{if}\ v^*_{\lambda}(z)\leq x\leq u^*_{\lambda}(z)\\
			f(z,u^*_{\lambda}(z))+u^*_{\lambda}(z)^{p-1}&\mbox{if}\ u^*_{\lambda}(z)<x
		\end{array}\right.\label{eq97}\\
		&&e(z,x)=\left\{\begin{array}{ll}
			\beta(z,v^*_{\lambda}(z))&\mbox{if}\ x<v^*_{\lambda}(z)\\
			\beta(z,x)&\mbox{if}\ v^*_{\lambda}(z)\leq x\leq u^*_{\lambda}(z)\\
			\beta(z,u^*_{\lambda}(z))&\mbox{if}\ u^*_{\lambda}(z)<x
		\end{array}\right.\ \mbox{for all}\ (z,x)\in\partial\Omega\times\RR\,.\label{eq98}
	\end{eqnarray}
	
	We set $K(z,x)=\int^x_0k(z,s)ds$ and $E(z,x)=\int^x_0e(z,s)ds$ and consider the $C^1$-functional $\gamma_{\lambda}:W^{1,p}(\Omega)\rightarrow\RR$ defined by
	 $$\gamma_{\lambda}(u)=\int_{\Omega}G(Du)dz+\frac{1}{p}||u||^p_p-\int_{\Omega}K(z,u)dz-\lambda\int_{\partial\Omega}E(z,u)d\sigma\ \mbox{for all}\ u\in W^{1,p}(\Omega).$$
	
	Also, we consider the positive and negative truncations of $k(z,\cdot),\ e(z,\cdot)$, that is, the Carath\'eodory functions
	$$k_{\pm}(z,x)=k(z,\pm x^{\pm})\ \mbox{and}\ e_{\pm}(z,x)=e(z,\pm x^{\pm}).$$
	
	We set $K_{\pm}(z,x)=\int^x_0k_{\pm}(z,s)ds$ and $E_{\pm}(z,x)=\int^x_0e_{\pm}(z,s)ds$ and consider the $C^1$ - functionals $\gamma^{\pm}_{\lambda}:W^{1,p}(\Omega)\rightarrow\RR$ defined by
	\begin{eqnarray*}
		 &&\gamma^{\pm}_{\lambda}(u)=\int_{\Omega}G(Du)dz+\frac{1}{p}||u||^p_p-\int_{\Omega}K_{\pm}(z,u)dz-\lambda\int_{\partial\Omega}E_{\pm}(z,u)d\sigma\ \mbox{for all}\ u\in W^{1,p}(\Omega).
	\end{eqnarray*}
	
	\begin{claim}\label{cl1}
		$K_{\gamma_{\lambda}}\subseteq[v^*_{\lambda},u^*_{\lambda}],
 K_{\gamma^+_{\lambda}}=\{0,u^*_{\lambda}\},\ K_{\gamma^-_{\lambda}}=\{0,v^*_{\lambda}\}$.
	\end{claim}
	
	Suppose that $u\in K_{\gamma_{\lambda}}$. Then
	\begin{equation}\label{eq99}
		\left\langle A(u),h\right\rangle+\int_{\Omega}|u|^{p-2}uhdz=\int_{\Omega}k(z,u)hdz+\lambda\int_{\partial\Omega}e(z,u)hd\sigma\ \mbox{for all}\ h\in W^{1,p}(\Omega).
	\end{equation}
	
	In (\ref{eq99}) first we choose $h=(u-u^*_{\lambda})^+\in W^{1,p}(\Omega)$. We obtain
	\begin{eqnarray*}
		&&\left\langle A(u),(u-u^*_{\lambda})^+\right\rangle+\int_{\Omega}|u|^{p-2}u(u-u^*_{\lambda})^+dz\\
		 &=&\int_{\Omega}[f(z,u^*_{\lambda})+(u^*_{\lambda})^{p-1}](u-u^*_{\lambda})^+dz+\lambda\int_{\partial\Omega}\beta(z,u^*_{\lambda})(u-u^*_{\lambda})^+d\sigma\ (\mbox{see (\ref{eq97}), (\ref{eq98})})\\
		&=&\left\langle A(u^*_{\lambda}),(u-u^*_{\lambda})^+\right\rangle+\int_{\Omega}(u^*_{\lambda})^{p-1}(u-u^*_{\lambda})^+dz\ (\mbox{since}\ u^*_{\lambda}\in S_+(\lambda)),\\
		\Rightarrow&&\left\langle A(u)-A(u^*_{\lambda}),(u-u^*_{\lambda})^+\right\rangle+\int_{\Omega}[|u|^{p-2}u-(u^*_{\lambda})^{p-1}](u-u^*_{\lambda})^+dz\leq 0,\\
		\Rightarrow&&u\leq u^*_{\lambda}.
	\end{eqnarray*}
	
	Similarly, if in (\ref{eq99}) we choose $h=(v^*_{\lambda}-u)^+\in W^{1,p}(\Omega)$, then we can show that
	$$v^*_{\lambda}\leq u.$$
	
	So, we have proved that
	\begin{eqnarray*}
		&&u\in[v^*_{\lambda},u^*_{\lambda}],\\
		&\Rightarrow&K_{\gamma_{\lambda}}\subseteq[v^*_{\lambda},u^*_{\lambda}].
	\end{eqnarray*}
	
	Similarly, we show that
	$$K_{\gamma^+_{\lambda}}\subseteq[0,u^*_{\lambda}]\ \mbox{and}\ K_{\gamma^-_{\lambda}}\subseteq[v^*_{\lambda},0].$$
	
	The extremality of the constant sign solutions $u^*_{\lambda}$ and $v^*_{\lambda}$, implies that
	$$K_{\gamma^+_{\lambda}}=\{0,u^*_{\lambda}\}\ \mbox{and}\ K_{\gamma^-_{\lambda}}=\{0,v^*_{\lambda}\}.$$
	
	This proves Claim \ref{cl1}.
	
	\begin{claim}\label{cl2}
		$u^*_{\lambda}\in D_+$ and $v^*_{\lambda}\in-D_+$ are local minimizers of $\gamma_{\lambda}$.
	\end{claim}
	
	Corollary \ref{cor3} and (\ref{eq97}), (\ref{eq98}) imply that $\gamma^+_{\lambda}$ is coercive. Also, it is sequentially weakly lower semicontinuous. So, we can find $\tilde{u}_{\lambda}\in W^{1,p}(\Omega)$ such that
	\begin{equation}\label{eq100}
		\gamma^+_{\lambda}(\tilde{u}_{\lambda})=\inf[\gamma_{\lambda}(u):u\in W^{1,p}(\Omega)].
	\end{equation}
	
	As before (see the proof of Proposition \ref{prop12}), since $q<\tau<p<\eta$, we have
	\begin{eqnarray}\label{eq101}
	&&\gamma^+_{\lambda}(\tilde{u}_{\lambda})<0=\gamma^+_{\lambda}(0),\nonumber\\
	&\Rightarrow&\tilde{u}_{\lambda}\neq 0.
	\end{eqnarray}
	
	From (\ref{eq100}) we have $\tilde{u}_{\lambda}\in K_{\gamma^+_{\lambda}}$. Then Claim \ref{cl1} and (\ref{eq101}) imply that
	$$\tilde{u}_{\lambda}=u^*_{\lambda}\in D_+.$$
	
	Note that
	\begin{eqnarray*}
		&&\gamma_{\lambda}|_{C_+}=\gamma^+_{\lambda}|_{C_+}\\
		&\Rightarrow&u^*_{\lambda}\ \mbox{is a local}\ C^1(\overline{\Omega})-\mbox{minimizer of}\ \gamma_{\lambda},\\
		&\Rightarrow&u^*_{\lambda}\ \mbox{is a local}\ W^{1,p}(\Omega)-\mbox{minimizer of}\ \gamma_{\lambda}\ (\mbox{see Proposition \ref{prop4}}).
	\end{eqnarray*}
	
	Similarly, for $v^*_{\lambda}\in-D_+$, using this time the functional $\gamma^-_{\lambda}$.
	This proves Claim \ref{cl2}.
	
	Without any loss of generality we may assume that
	$$\gamma_{\lambda}(v^*_{\lambda})\leq\gamma_{\lambda}(u^*_{\lambda}).$$
	
	The reasoning is similar if the opposite inequality holds.
	
	We assume that $K_{\gamma_{\lambda}}$ is finite. Otherwise, on account of Claim \ref{cl1} and  (\ref{eq97}), (\ref{eq98}), we already have an infinity of nodal solutions in $C^1(\overline{\Omega})$ (by the nonlinear regularity theory of Lieberman \cite{21}). Then since $u^*_{\lambda}\in D_+$ is a local minimizer of $\gamma_{\lambda}$ (see Claim \ref{cl2}), we can find $\rho\in(0,1)$ so small that
	\begin{equation}\label{eq102}
		\gamma_{\lambda}(v^*_{\lambda})\leq \gamma_{\lambda}(u^*_{\lambda})<\inf[\gamma_{\lambda}(u):||u-u^*_{\lambda}||=\rho]=m_{\lambda}
	\end{equation}
	(see Aizicovici, Papageorgiou and Staicu \cite{1}, proof of Proposition 29).
	
	The functional $\gamma_{\lambda}$ is coercive (see (\ref{eq97}), (\ref{eq98})). So, we have that
	\begin{equation}\label{eq103}
		\gamma_{\lambda}\ \mbox{satisfies the $C$-condition}
	\end{equation}
	(see Papageorgiou and Winkert \cite{32}). Then (\ref{eq102}), (\ref{eq103}) permit the use of Theorem \ref{th1} (the mountain pass theorem). So, we can find $y_0\in W^{1,p}(\Omega)$ such that
	\begin{equation}\label{eq104}
		y_0\in K_{\gamma_{\lambda}}\ \mbox{and}\ m_{\lambda}\leq \gamma_{\lambda}(y_0).
	\end{equation}
	
	From (\ref{eq102}) and (\ref{eq104}) we see that
	\begin{equation}\label{eq105}
		y_0\notin\{u^*_{\lambda},v^*_{\lambda}\}\ \mbox{and}\ y_0\in C^1(\overline{\Omega})\ (\mbox{nonlinear regularity theory}).
	\end{equation}
	
	Moreover, Corollary 6.81, p. 168 of Motreanu, Motreanu and Papageorgiou \cite{23} implies that
	\begin{equation}\label{eq106}
		C_1(\gamma_{\lambda},y_0)\neq 0.
	\end{equation}
	\begin{claim}\label{cl3}
		$C_k(\gamma_{\lambda},0)=C_k(\varphi_{\lambda},0)$ for all $k\in\NN_0$.
	\end{claim}
	
	We consider the homotopy $h(t,u)$ defined by
	$$h(t,u)=(1-t)\varphi_{\lambda}(u)+t\gamma_{\lambda}(u)\ \mbox{for all}\ (t,u)\in[0,1]\times W^{1,p}(\Omega).$$
	
	Suppose that we could find $\{t_n\}_{n\geq 1}\subseteq[0,1]$ and $\{u_n\}_{n\geq 1}\subseteq W^{1,p}(\Omega)$ such that
	\begin{equation}\label{eq107}
		t_n\rightarrow t\ \mbox{in}\ [0,1],u_n\rightarrow 0\ \mbox{in}\ W^{1,p}(\Omega)\ \mbox{and}\ h'_u(t_n,u_n)=0\ \mbox{for all}\ n\in\NN.
	\end{equation}
	
	From the equality in (\ref{eq107}), we have
	\begin{eqnarray*}
		&&\left\langle A(u_n),h\right\rangle+t_n\int_{\Omega}|u_n|^{p-2}u_nhdz\\
		 &&=(1-t_n)\int_{\Omega}f(z,u_n)hdz+t_n\int_{\Omega}k(z,u_n)hdz+\lambda\int_{\partial\Omega}[(1-t_n)\beta(z,u_n)+t_ne(z,u_n)]d\sigma\\
		&&\mbox{for all}\ n\in\NN,\ \mbox{all}\ h\in W^{1,p}(\Omega).
	\end{eqnarray*}
	
	It follows (see Papageorgiou and R\u adulescu \cite{28}) that
	\begin{eqnarray}\label{eq108}
		&&\left\{\begin{array}{l}
			-{\rm div}\,a (Du_n(z))+t_n|u_n(z)|^{p-2}u_n(z)=(1-t_n)f(z,u_n(z))+t_nk(z,u_n(z))\\
			\mbox{for almost all}\ z\in\Omega,\\
			\frac{\partial u_n}{\partial n_a}=\lambda[(1-t_n)\beta(z,u_n)+t_ne(z,u_n)]\\
			\mbox{on}\ \partial\Omega,\ n\in\NN.
		\end{array}\right\}
	\end{eqnarray}
	
	From (\ref{eq107}), (\ref{eq108}), we see that we can find $c_{45}>0$ such that
	\begin{equation}\label{eq109}
		||u_n||_{\infty}\leq c_{45}\ \mbox{for all}\ n\in\NN
	\end{equation}
	(see \cite{19} and \cite{30}). This $L^{\infty}$-bound permits the use of the nonlinear regularity theory of Lieberman \cite{21}, hence there exist $\alpha\in(0,1)$ and $c_{46}>0$ such that
	\begin{equation}\label{eq110}
		u_n\in C^{1,\alpha}(\overline{\Omega})\ \mbox{and}\ ||u_n||_{C^{1,\alpha}(\overline{\Omega)}}\leq c_{46}\ \mbox{for all}\ n\in\NN.
	\end{equation}
	
	From (\ref{eq107}), (\ref{eq110}) and the compact embedding of $C^{1,\alpha}(\overline{\Omega})$ into $C^1(\overline{\Omega})$, we have
	\begin{eqnarray}\label{eq111}
		&&u_n\rightarrow 0\ \mbox{in}\ C^1(\overline{\Omega})\nonumber,\\
		&\Rightarrow&u_n\in[v^*_{\lambda},u^*_{\lambda}]\ \mbox{for all}\ n\geq n_0.
	\end{eqnarray}
	
It follows	from (\ref{eq97}), (\ref{eq98}), (\ref{eq108}), (\ref{eq111}) that
	$$u_n\in K_{\gamma_{\lambda}}\ \mbox{for all}\ n\geq n_0,$$
	which contradicts the assumption that $K_{\gamma_{\lambda}}$ is finite. So, (\ref{eq107}) cannot occur and we can use Theorem 5.2 of Corvellec and Hantoute \cite{8} (the homotopy invariance of critical groups) and obtain
	\begin{eqnarray*}
		&&C_k(h(0,\cdot),0)=C_k(h(1,\cdot),0)\ \mbox{for all}\ k\in\NN_0,\\
		&\Rightarrow&C_k(\varphi_{\lambda},0)=C_k(\gamma_{\lambda},0)\ \mbox{for all}\ k\in\NN_0.
	\end{eqnarray*}
	
	This proves Claim \ref{cl3}.
	
	From Claim \ref{cl3} and Proposition \ref{prop18}, we have
	\begin{equation}\label{eq112}
		C_k(\gamma_{\lambda},0)=0\ \mbox{for all}\ k\in\NN_0.
	\end{equation}
	
	Comparing (\ref{eq106}) and (\ref{eq112}), we see that
	\begin{eqnarray*}
		&&y_0\neq 0,\\
		&\Rightarrow&y_0\in[v^*_{\lambda},u^*_{\lambda}]\cap C^1(\overline{\Omega})\ \mbox{is a nodal solution of \eqref{eqP} (see (\ref{eq105}))}.
	\end{eqnarray*}
\end{proof}

Summarizing the situation for problem \eqref{eqP}, we can state the following multiplicity theorem.
\begin{theorem}\label{th20}
	If hypotheses $H(a),H(f)',H(\beta)$ hold, then there exists $\lambda_0>0$ such that for every $\lambda\in(0,\lambda_0)$ problem \eqref{eqP} has at least five nontrivial smooth solutions
	$$u_0,\hat{u}\in D_+,\ v_0,\hat{v}\in-D_+,\ y_0\in C^1(\overline{\Omega})\ \mbox{nodal}.$$
	Moreover, for every $\lambda\in(0,\lambda_0)$, problem \eqref{eqP} has extremal constant sign solutions
	$$u^*_{\lambda}\in D_+\ \mbox{and}\ v^*_{\lambda}\in-D_+$$
	such that $y_0\in[v^*_{\lambda},u^*_{\lambda}]\cap C^1(\overline{\Omega})$ and the map $\lambda\mapsto u^*_{\lambda}$ is
	\begin{itemize}
		\item strictly increasing (that is, $\mu<\lambda\Rightarrow u^*_{\lambda}-u^*_{\mu}\in {\rm int}\, \hat{C}_+$),
		\item left continuous from $(0,\lambda_0)$ into $C^1(\overline{\Omega})$,
	\end{itemize}
	while the map $\lambda\mapsto v^*_{\lambda}$ is
	\begin{itemize}
		\item	strictly decreasing (that is, $\mu<\lambda\Rightarrow v^*_{\mu}-v^*_{\lambda}\in {\rm int}\,\hat{C}_+$),
		\item right continuous.
	\end{itemize}
\end{theorem}

In the next section, we show that in the semilinear case, we can improve this theorem and produce a sixth nontrivial smooth solution $\hat{y}$, but we cannot provide any sign information for it.

\section{Semilinear Problem}

In this section we deal with the semilinear problem
\begin{equation}
	\left\{\begin{array}{ll}
		-\Delta u(z)=f(z,u(z))&\mbox{in}\ \Omega,\\
		\frac{\partial u}{\partial n}=\lambda\beta(z,u)&\mbox{on}\ \partial\Omega\,.
	\end{array}\right\}\tag{$S_{\lambda}$}\label{eqS}
\end{equation}

We strengthen the regularity hypotheses on the reaction term $f(z,\cdot)$ and on the boundary (source) term $\beta(z,\cdot)$ and by using Morse theory we are able to generate a sixth nontrivial smooth solution. However, we cannot provide any sign information for this new solution.
	
	In this case the energy (Euler) functional of problem \eqref{eqS} is $\varphi_{\lambda}:H^1(\Omega)\rightarrow\RR$  defined by
	$$\varphi_{\lambda}(u)=\frac{1}{2}||Du||^2_2-\int_{\Omega}F(z,u)dz-\lambda\int_{\partial\Omega}B(z,u)d\sigma\ \mbox{for all}\ u\in H^1(\Omega).$$
	
	Hypotheses $H(f)''$ and $H(\beta)'$ imply that $\varphi_{\lambda}\in C^2(H^1(\Omega)\backslash\{0\})$. Under these hypotheses we can show that problem \eqref{eqS} has six nontrivial smooth solutions for all small $\lambda>0$.
\begin{theorem}\label{th21}
	If hypotheses $H(f)'',\ H(\beta)'$ hold, then we can find $\lambda_0>0$ such that for every $\lambda\in(0,\lambda_0)$ problem \eqref{eqS} has at least six nontrivial smooth solutions
		\begin{eqnarray*}
				&&u_0,\hat{u}\in D_+,\ v_0,\hat{v}\in-D_+\\
				&&y_0\in C^1(\overline{\Omega})\ \mbox{nodal and}\ \hat{y}\in C^1(\overline{\Omega}).
		\end{eqnarray*}
\end{theorem}
\begin{proof}
	From Theorem \ref{th20}, we know that we can find $\lambda_0>0$ such that for all $\lambda\in(0,\lambda_0)$ problem \eqref{eqS} has five nontrivial smooth solutions
	$$u_0,\hat{u}\in D_+,\ v_0,\hat{v}\in-D_+\ \mbox{and}\ y_0\in[v_0,u_0]\cap C^1(\overline{\Omega})\ \mbox{nodal}.$$
	
	From the proof of Proposition \ref{prop12}, we know that $u_0\in D_+$ and $v_0\in-D_+$ are local minimizers of $\varphi_{\lambda}$ and so we have
	\begin{equation}\label{eq113}
		C_k(\varphi_{\lambda},u_0)=C_k(\varphi_{\lambda},v_0)=\delta_{k,0}\ZZ\ \mbox{for all}\ k\in\NN_0.
	\end{equation}
	
	Let $\rho=\max\{||u_0||_{\infty},||v_0||_{\infty}\}$ and let $\hat{\xi}_{\rho}>0$ be as postulated by hypothesis $H(f)(iv)$. We have
	\begin{eqnarray*}
		&-\Delta y_0(z)+\hat{\xi}_{\rho}y_0(z)&=f(z,y_0(z))+\hat{\xi}_{\rho}y_0(z)\\
		&&\leq f(z,u_0(z))+\hat{\xi}_{\rho}u_0(z)=-\Delta u_0(z)+\hat{\xi}_{\rho}u_0(z)\ \mbox{for almost all}\ z\in\Omega,\\
		&\Rightarrow\Delta(u_0-y_0)(z)&\leq\hat{\xi}_{\rho}(u_0-y_0)(z)\ \mbox{for almost all}\ z\in\Omega,\\
		&\Rightarrow u_0-y_0\in D_+&(\mbox{by the strong maximum principle}).
	\end{eqnarray*}
	
	Similarly, we show that
	$$y_0-v_0\in D_+.$$
	
	Therefore we can assert that
	\begin{equation}\label{eq114}
		y_0\in {\rm int}_{C^1(\overline{\Omega})}[v_0,u_0].
	\end{equation}
	
	Keeping the notation of the previous section (see the proof of Proposition \ref{prop19}), and assuming without any loss of generality that $u_0,v_0$ are extremal constant sign solutions (see Proposition \ref{prop15}), we have
	$$\gamma_{\lambda}|_{[v_0,u_0]}=\varphi_{\lambda}|_{[v_0,u_0]}.$$
	
	Then it follows from (\ref{eq114}) that
	\begin{eqnarray*}
		&&C_k(\gamma_{\lambda}|_{C^1(\overline{\Omega})},y_0)=C_k(\varphi_{\lambda}|_{C^1(\overline{\Omega})},y_0)\ \mbox{for all}\ k\in\NN_0,\\
		&\Rightarrow&C_k(\gamma_{\lambda},y_0)=C_k(\varphi_{\lambda},y_0)\ \mbox{for all}\ k\in\NN_0
	\end{eqnarray*}
	(since $C^1(\overline{\Omega})$ is dense in $H^1(\Omega)$, see Chang \cite[p. 14]{5} and Palais \cite{25}),
	\begin{equation}\label{eq115}
		\Rightarrow C_1(\varphi_{\lambda},y_0)\neq 0
	\end{equation}
	(since $y_0$ is a critical point of mountain pass type of $\gamma_{\lambda}$).
	
	Since $\varphi_{\lambda}\in C^2(H^1(\Omega)\backslash\{0\})$, it follows from (\ref{eq115}) that
	\begin{equation}\label{eq116}
		C_k(\varphi_{\lambda},y_0)=\delta_{k,1}\ZZ\ \mbox{for all}\ k\in\NN_0
	\end{equation}
	(see Motreanu, Motreanu and Papageorgiou \cite{23}, Corollary 6.102, p. 177).
	
	Similarly, from the proof of Proposition \ref{prop12} and keeping the notion introduced there,
	\begin{eqnarray*}
		&&\hat{u}\in D_+\ \mbox{is a critical point of mountain pass type of}\ \hat{\varphi}^+_{\lambda},\\
		&&\hat{v}\in-D_+\ \mbox{is a critical point of mountain pass type of}\ \hat{\varphi}^-_{\lambda}.
	\end{eqnarray*}
	
	Since $\hat{\varphi}^+_{\lambda}|_{C_+}=\varphi_{\lambda}|_{C_+}$ and $\hat{\varphi}^-_{\lambda}|_{-C_+}=\varphi_{\lambda}|_{-C_+}$, as above we have
	\begin{equation}\label{eq117}
		C_k(\varphi_{\lambda},\hat{u})=C_k(\varphi_{\lambda},\hat{v})=\delta_{k,1}\ZZ\ \mbox{for all}\ k\in\NN_0.
	\end{equation}
	
	From Proposition \ref{prop18}, we know that
	\begin{equation}\label{eq118}
		C_k(\varphi_{\lambda},0)=0\ \mbox{for all}\ k\in\NN_0.
	\end{equation}
	
	Finally, hypothesis $H(f)''(ii)$ implies that
	\begin{equation}\label{eq119}
		C_k(\varphi_{\lambda},\infty)=0\ \mbox{for all}\ k\in\NN_0
	\end{equation}
	(see Papageorgiou and R\u adulescu \cite[Proposition 13]{29}). Suppose that $$K_{\varphi_{\lambda}}=\{0,u_0,\hat{u},v_0,\hat{v},y_0\}.$$ Then using (\ref{eq113}), (\ref{eq116}), (\ref{eq117}), (\ref{eq118}), (\ref{eq119}) and the Morse relation (see (\ref{eq4})) with $t=-1$, we obtain
	\begin{eqnarray*}
		&&2(-1)^0+2(-1)^1+(-1)^1=0,\\
		&\Rightarrow&(-1)^1=0,\ \mbox{a contradiction}
	\end{eqnarray*}
	
	So, there exists $\hat{y}\in K_{\varphi_{\lambda}},\hat{y}\notin\{0,u_0,\hat{u},v_0,\hat{v},y_0\}$. Then $\hat{y}$ is the sixth nontrivial solution of \eqref{eqS} and $\hat{y}\in C^1(\overline{\Omega})$ (by regularity theory).
\end{proof}
	
 \medskip
{\bf Acknowledgments.} The authors wish to thank the two knowledgeable referees for their comments, remarks and constructive criticism.	This research was supported by the Slovenian Research Agency grants P1-0292, J1-8131, N1-0083, and N1-0064. V.D. R\u adulescu acknowledges the support through a grant of the Romanian Ministry of Research and Innovation, CNCS - UEFISCDI, project number PN-III-P4-ID-PCE-2016-0130, within PNCDI III.

\end{document}